\documentclass[11pt]{article}
\usepackage{amsmath}
\usepackage{amsfonts}
\usepackage{amssymb}
\usepackage{multirow}
\usepackage{mathrsfs}
\usepackage[pagewise]{lineno}
\usepackage{textcomp}
\usepackage{tikz}

\usepackage{algorithmic}
\usepackage{algorithm}

\usepackage{xcolor}
\usepackage{graphicx}
\usepackage{float}
\usepackage{booktabs}
\usepackage{longtable}
\usepackage{placeins}
\usepackage{booktabs,adjustbox}
\usepackage{array}
\usepackage{caption}
\renewcommand{\Box}{\framebox{\rule{0.3em}{0.0em}}}

\newtheorem{thm}{Theorem}[section]
\newtheorem{lema}{Lemma}[section]

\newtheorem{ex}{Example}[section]

\newtheorem{defi}{Definition}[section]

\newcommand{\z}{{ z}}

\newcommand{\bgeqn}{\begin{eqnarray}}
\newcommand{\edeqn}{\end{eqnarray}}
\newcommand{\bgeq}{\begin{eqnarray*}}
\newcommand{\edeq}{\end{eqnarray*}}
\newcommand{\bec}{\begin{center}}
\newcommand{\enc}{\end{center}}

\newcommand{\be}{\begin{equation}}
\newcommand{\ee}{\end{equation}}

\newtheorem{remark}{Remark}[section]
\renewcommand{\Box}{\hfill \rule{2.3mm}{2.3mm}}

\makeatletter

\@addtoreset{equation}{section} \makeatother
\newenvironment{proof}{\noindent{\bf Proof. }}{\hfill $\Box$\medskip}

\title{\Large \bf
Augmented Lagrangian Method for Mathematical Programs with Second-Order Cone Complementarity Constraints\thanks{\baselineskip 9pt
This work was supported in part by the National Natural Science Foundation of China (Grant No. 12571324) and Natural Science Foundation of Henan Province (Grant No. 262300421849) and Henan Provincial Selective Research Funding Program for Returned Scholars Studying Abroad
(Grant No. HNLX202609).
}
}
\author{
Yan-Chao Liang\thanks{\baselineskip 9.5pt School of Mathematics and Statistics, Henan Normal University, Xinxiang 453007, China. E-mail: liangyanchao83@163.com.},  \
Chen-Yuan Zhu\thanks{\baselineskip 9.5pt School of Mathematics and Statistics, Henan Normal University, Xinxiang 453007, China. E-mail: zhuchenyuan08@163.com.}, \
Sheng-Jie Zhang\thanks{\baselineskip 9.5pt School of Mathematics and Statistics, Henan Normal University, Xinxiang 453007, China. E-mail: 2201183015@stu.htu.edu.cn.}, \
Gui-Hua Lin\thanks{\baselineskip 9.5pt School of Management, Shanghai University, Shanghai 200444, China. E-mail:  guihualin@shu.edu.cn.}
\ and
Xide Zhu\thanks{\baselineskip 9.5pt Corresponding author, School of Management, Shanghai University, Shanghai 200444, China. E-mail: xidezhu@shu.edu.cn.}}
\date{}

\begin{document}
\maketitle

{\vspace{13pt} \noindent{\bf Abstract.}
This paper investigates mathematical programs with second-order cone complementarity constraints (SOCMPCCs), which extend classical mathematical programs with complementarity constraints (MPCCs) by incorporating second-order cone structures. SOCMPCCs present significant theoretical and computational challenges, primarily due to the failure of standard constraint qualifications (such as Robinson's constraint qualification) at all feasible points. This difficulty hinders the direct application of classical nonlinear programming theories and algorithms. Motivated by the success of the augmented Lagrangian method (ALM) in solving MPCCs, we explore its extension to SOCMPCCs. The ALM, known for its matrix-free implementation and strong local convergence properties, is well suited for handling the intricate interplay between complementarity and second-order cone constraints. In this paper, we propose a tailored ALM algorithm framework for SOCMPCCs and establish its feasibility and convergence properties. We show that, under bounded ALM penalty parameters or bounded augmented Lagrangian functions, the generated sequence converges to feasible points of the SOCMPCC. 
Moreover, under the feasibility and an additional SOCMPCC-weak-Robinson's constraint qualification, we prove convergence to K-stationary points, which constitute a fundamental optimality condition for SOCMPCCs.
Numerical experiments, including both illustrative examples and high-dimensional problems, are conducted to demonstrate the effectiveness and practical applicability of the proposed algorithm in addressing the challenges inherent in SOCMPCCs.

\vspace{13pt}\noindent{\bf Keywords.}
Mathematical programs with second-order cone complementary constraints, augmented Lagrangian method, K-stationary point, {SOCMPCC-weak-Robinson's constraint qualification}.


\section{Introduction}
Mathematical programs with complementarity constraints (MPCCs) arise in a wide range of applications, including engineering design, economic equilibrium modeling, transportation science, and multi-level game theory; see \cite{Luo Pang Ralph,Outrata Kocvara Zowe} and the references therein. When MPCCs are regarded as nonlinear programs with equality and inequality constraints, the standard Mangasarian–Fromovitz constraint qualification (MFCQ) fails to hold at any feasible point \cite{Ye Zhu Zhu}.
This poses significant challenges for directly applying classical nonlinear programming theories and algorithms.
Second-order cone programs (SOCPs) are another important optimization problems, with applications in antenna array design, finite impulse response filter design, and portfolio optimization \cite{F. Alizadeh and D. Goldfarb}.
Second-order cone complementarity problems (SOCCPs) generalize classical complementarity problems and includes the Karush–Kuhn–Tucker (KKT) conditions of SOCPs as a special case. SOCCPs have received considerable attention in the literature; see \cite{Chen Chen Tseng,Chen-Sun-Sun,fukushima-luo-tseng2001,Hayashi-Yamashita-Fukushima2005} and the survey in \cite{Chen-Pan2012} for recent developments.

In this paper, we investigate mathematical programs with second-order cone complementarity constraints (SOCMPCCs), which extend classical MPCCs by incorporating second-order cone constraints. As shown in \cite{Y.-C. Liang Y.-W. Liu}, SOCMPCCs have important applications, including anti-quadratic programming over second-order cones, robust bilevel programming, and reformulations of MPCCs. However, Robinson's constraint qualification for SOCMPCCs fails at all feasible points \cite{Ye Zhou}.
Liang et al. \cite{Liang Zhu Lin } provided expressions for the regular and limiting normal cones of the second-order cone complementarity set, although gaps exist in their treatment, particularly at boundary points. These gaps were later addressed by Ye and Zhou \cite{Ye Zhou, Ye JJ Zhou J}, who established exact characterizations of these normal cones and derived various necessary optimality conditions under suitable constraint qualifications. They showed that unlike in MPCCs, the classical KKT conditions for SOCMPCCs generally do not imply strong stationarity, except when the dimension of each cone is at most two. They also showed that reformulating MPCCs as SOCMPCCs can also lead to weaker necessary conditions.
Using Jordan algebra techniques, Zhu et al. \cite{Zhu Zhang Zhou Yang } proposed an alternative formulation of SOCMPCC and showed that the classical KKT conditions are consistent with strong stationarity under this reformulation. They also introduced an approximation method that ensures Clarke stationarity under a linear independence constraint qualification, which can be strengthened to strong stationarity under a strict complementarity condition.

Several researchers have proposed smoothing methods to handle SOCMPCCs. Yan and Fukushima \cite{Yan Fukushima } developed a smoothing method for mathematical programs with symmetric cone complementarity constraints, including SOCMPCCs. Yamamura et al. \cite{Yamamura Okuno Hayashi Fukushima } analyzed linear SOCMPCCs and proposed a smoothing SQP method based on reformulation via natural residual functions. Zhu et al. \cite{zhu-pang-lin} and \cite{Zhu Zhang Zhou Yang } presented further smoothing and approximation strategies, demonstrating convergence to stationary points under appropriate conditions. Liang et al. \cite{Y.-C. Liang Y.-W. Liu} proposed several new constraint qualifications for SOCMPCCs, including the SOCMPCC relaxed constant positive linear dependence condition and other constant rank-type constraint qualifications. These new conditions are strictly weaker than the SOCMPCC linear independence constraint qualification and the nondegeneracy condition. Furthermore, the relationships among various existing constraint qualifications for SOCMPCCs are systematically investigated.
{Recently, Liang et al. \cite{Liang-Zhu-Lin 2026 JOTA} proposed several new weak constraint qualifications for SOCMPCCs by using spectral decomposition and cone slicing techniques. They further proved that these conditions are sufficient to guarantee that a local minimizer of SOCMPCC is a
K-stationary point.}

The augmented Lagrangian method (ALM) was first proposed by Hestenes \cite{Hestenes M-R} and Powell \cite{Powell M J D} for equality-constrained optimization problems and was later extended to nonlinear programming with inequality constraints by Rockafellar \cite{Rockafellar1970}.
Shapiro and Sun \cite{Shapiro} studied the properties of the augmented Lagrangian in cone-constrained optimization problems.
ALM can be implemented in a matrix-free manner and possesses strong local convergence guarantees.
These features make it particularly effective for solving large-scale problems.
Various variants of the classical ALM have been developed, including the protected ALM \cite{Andreani Birgin Martinez Schuverdt ,Birgin Martinez ,Kanzow Steck }.
In recent years, ALM has been applied to MPCCs and SOCPs.
Yang and Huang \cite{Yang Huang } proposed an ALM for MPCCs using the {Fischer-Burmeister} function.
Huang et al. \cite{Huang Yang Teo } applied a partial ALM that incorporates the complementarity constraints into the objective function.
Andreani et al. \cite{Andrean Secchin Silva } employed a second-order ALM, originally designed for nonlinear programming, to address MPCCs.
Guo and Deng \cite{Guo Lei and Z Deng } proposed a method that excludes the complementarity constraints from the augmented Lagrangian function and instead handles them directly using a projection-gradient approach.
Liu and Zhang \cite{Liu-zhang-2007} established the convergence of ALM for SOCPs.
Building on this, Hang et al. \cite{Hang2021} used second-order sufficiency to establish a uniform second-order growth condition for ALM in SOCPs.
Furthermore, Andreani et al. \cite{Andreani-MOR2022} introduced new sequential optimality conditions within a general nonlinear conic programming framework.
They demonstrated that feasible limit points of sequences generated by ALM satisfy the so-called approximate gradient projection optimality condition and, under an additional smoothness assumption, the complementary approximate KKT condition.

ALM has demonstrated high efficiency in solving MPCCs, which motivates us to explore its potential application to SOCMPCCs.
Given its successful application to MPCCs, it is natural to consider extending ALM to the more complex setting of SOCMPCCs.
SOCMPCCs not only involve the complementarity constraints present in MPCCs but also incorporate second-order cone constraints, thereby increasing the problem's complexity and difficulty. Nevertheless, the strengths of ALM in handling complex constraint structures make its application to SOCMPCCs a promising direction for further investigation.

The main contributions of this paper are summarized as follows:\\[-6mm]
\begin{itemize}
\item We propose a tailored ALM algorithm framework for SOCMPCCs and establish its feasibility and convergence properties. In particular, we show that, under bounded ALM penalty parameters or bounded augmented Lagrangian functions, the generated sequence converges to feasible points of the original SOCMPCC.\\[-6mm]

\item {We further show that, when ALM penalty parameters are unbounded, every limit point of the generated sequence is a stationary point of an associated unconstrained optimization problem. Under the feasibility and { additional SOCMPCC-weak-Robinson's constraint qualification}, we establish convergence to K-stationary points of the SOCMPCC.}\\[-6mm]

\item Numerical experiments, including both illustrative examples and high-dimensional test problems, demonstrate that the proposed ALM exhibits improved solution accuracy compared with existing smoothing methods and is effective for large-scale SOCMPCCs.
\end{itemize}

The rest of this paper is organized as follows.
Section 2 presents fundamental concepts and results related to SOCPs and SOCMPCCs.
In Section 3, we propose the algorithm framework of ALM for SOCMPCCs, and provide an analysis of the feasibility and convergence of the proposed algorithm.
Section 4 reports numerical experiments conducted to validate the effectiveness of the algorithm.
Section 5 concludes the paper.

%%%%%%%%%%%%%%%%%%%%%%%%%%%%%%%%%%%%%%%%%
%%%%%%%%%%%%%%%%%%%%%%%%%%%%%%%%%%%%%%%%%
%%%%%%%%%%%%%%%%%%%%%%%%%%%%%%%%%%%%%%%%%
%%%%%%%%%%%%%%%%%%%%%%%%%%%%%%%%%%%%%%%%%
%%%%%%%%%%%%%%%%%%%%%%%%%%%%%%%%%%%%%%%%%

\section{Preliminaries}

Throughout this paper, we adopt the following notations, which are standard in optimization and variational analysis.
Let $\mathbb{N}$ denote the set of natural numbers (with $0 \in \mathbb{N}$), and let $\mathbb{R}^n$ represent the $n$-dimensional real Euclidean space.
We define $\mathbb{R}_+ := [0, \infty)$ and $\mathbb{R}_{++} := (0, \infty)$.
For any vector $a \in \mathbb{R}^n$, we write $a_+ := \max\{a, 0\}$ and $a_- := \min\{a, 0\}$, where the `$\max$' and `$\min$' operators are understood componentwise.
For a vector $x = (x_1, \bar{x}) \in \mathbb{R} \times \mathbb{R}^{m-1}$, its reflection is defined as $\hat{x}: = (x_1, -\bar{x}) \in \mathbb{R} \times \mathbb{R}^{m-1}$.
For a differentiable mapping $f(x) : \mathbb{R}^n \rightarrow \mathbb{R}^m$, we denote its Jacobian matrix by { $\mathcal{J} f(x)$, and the transpose of the Jacobian by $\nabla f(x)$.}
We denote by $\mathbb{E}$ a finite-dimensional linear space equipped with an inner product, denoted by $\langle \cdot, \cdot \rangle$.
Given a closed convex set $C \subset \mathbb{R}^n$, the polar of $C$ is defined as $C^\circ := \{ \omega \in \mathbb{R}^n : \langle \omega, x \rangle \leq 0,\ \forall x \in C \}$, and ${(C^\circ)}^\circ = C$ holds.
The distance from a point $\omega \in \mathbb{R}^n$ to $C$ is denoted by $\text{dist}_{C}(\omega) := \min\{ \| \omega - x \| : x \in C \}$, and the orthogonal projection of $\omega$ onto $C$, denoted by $\Pi_{C}(\omega)$, is the point in $C$ at which this minimum is attained.

Let ${\cal K}$ be a single second-order cone of dimension $m$, defined by
\begin{eqnarray*}
\mathcal{K}:=\big\{(x_1,\bar{x})\in \mathbb{R}\times\mathbb{R}^{m-1}\ |\  x_1\geq\| \bar{x}\|\big\},
\end{eqnarray*}
where $\|\cdot\|$ denotes the Euclidean norm. The topological interior and boundary of $\mathcal{K}$, denoted by ${\rm int}{\mathcal{K}}$ and ${\rm bd}{\mathcal{K}}$ respectively, are defined as
\begin{eqnarray*}
&&{\rm int}{\mathcal{K}}:=\{x=(x_1,\bar{x})\in \mathbb{R}\times\mathbb{R}^{m-1} |\ x_1>\| \bar{x}\| \},\\
&&{\rm bd}{\mathcal{K}}:=\{x=(x_1,\bar{x})\in \mathbb{R}\times\mathbb{R}^{m-1} |\ x_1=\| \bar{x}\| \}.
\end{eqnarray*}
It is well known that ${\mathcal{K}}^\circ=-\mathcal{K}$. In \cite{fukushima-luo-tseng2001}, a formula for the projection onto the cone ${\mathcal{K}}$ is given.
Following \cite{fukushima-luo-tseng2001}, every $x=(x_1,\bar{x})\in \mathbb{R}\times\mathbb{R}^{m-1}$ can be decomposed as
\begin{eqnarray*}
x=\mu_1(x)c_1(x)+\mu_2(x)c_2(x),
\end{eqnarray*}
where $\mu_j\in\mathbb{R}$ and $c_j\in{\mathcal{K}}$ for $j\in\{1,2\}$ are given by
\begin{eqnarray}\label{decomposition}
\mu_j(x)=x_1+(-1)^j||\bar{x}||,~~c_j(x)=\left\{
\begin{aligned}
&\tfrac{1}{2}(1,(-1)^j\tfrac{\bar{x}}{||\bar{x}||}), &~~~{\rm if }~~ \bar{x}\neq0,\\
&\tfrac{1}{2}(1,(-1)^j\omega),~~~&~~~{\rm if}~~ \bar{x}=0,
\end{aligned}
\right.
\end{eqnarray}
where $\omega$ is any unit vector in $\mathbb{R}^{m-1}$.
Then, for every $x=(x_1,\bar{x})\in\mathbb{R}\times\mathbb{R}^{m-1}$, the projection onto ${\mathcal{K} }$ is given by
$$
\Pi_{\mathcal{K}}(x):=\max\{\mu_1(x),0\}c_1(x)+\max\{\mu_2(x),0\}c_2(x).
$$
Using the identity $\mathcal{K} ^\circ = -\mathcal{K}$, we obtain the projection onto the polar cone:
\begin{eqnarray*}
\Pi_{\mathcal{K}^\circ}(x):=\min\{\mu_1(x),0\}c_1(x)+\min\{\mu_2(x),0\}c_2(x).
\end{eqnarray*}

In this paper, we study the following mathematical program with second-order cone complementarity constrains, denoted by {SOCMPCC}:
\begin{eqnarray}\label{SOCMPCC}
\min   & & f(z) \nonumber\\
{\rm SOCMPCC}~~~~\qquad\qquad\mbox{s.t.} & & g(z)\le 0,  \ h(z)=0,   \nonumber\qquad\qquad\qquad\qquad\qquad\qquad\\
&& G(z)\in \mathcal{K},\ H(z)\in \mathcal{K},\\
&&G(z)^T H(z)=0,\nonumber
\end{eqnarray}
where
$f:\mathbb{R}^n\rightarrow \mathbb{R}, ~g:\mathbb{R}^n\rightarrow \mathbb{R}^p,~h:\mathbb{R}^n\rightarrow \mathbb{R}^q,~G:\mathbb{R}^n\rightarrow \mathbb{R}^m,~H:\mathbb{R}^n\rightarrow \mathbb{R}^m$ are all continuously differentiable mappings,
$\mathcal{K}\subset{\mathbb{R}^m}$ is a nonempty closed convex cone defined as
$$
\mathcal{K}:= \mathcal{K}_{1}\times \mathcal{K}_{2}\times\cdots\times \mathcal{K}_{J},
$$
where $m=\sum\limits_{i=1}^{J}m_i$ and $\mathcal{K}_{i}$ is an $m_i$-dimensional second-order cone. $G(z):=(G_1(z),\cdots,G_J(z))$ and $H(z):=(H_1(z),\cdots,H_J(z))$ with $G_i(z)\in \mathcal{K}_{i} $ and $H_i(z)\in \mathcal{K}_{i}$ for $i=1,\cdots,J$.

For $\nu=(\nu_1,\nu_2,\cdots,\nu_J)\in{\mathbb{R}^m}$ with $\nu_i\in \mathbb{R}\times\mathbb{R}^{m_i-1}$, and with reference to the cone $\cal K$ in (\ref{SOCMPCC}), we have
$$
\Pi_{\mathcal{K}}(\nu)=(\Pi_{\mathcal{K}_1}(\nu_1),\cdots,\Pi_{\mathcal{K}_J}(\nu_J)),~~\Pi_{\mathcal{K}^\circ}(\nu)=(\Pi_{{\cal K}_1^\circ}(\nu_1),\cdots,\Pi_{\mathcal{K}_J^\circ}(\nu_J)).
$$

{
For any $x=(x_1,\bar{x})\in \mathbb{R}\times\mathbb{R}^{m-1}$ and $y=(y_1,\bar{y})\in \mathbb{R}\times\mathbb{R}^{m-1}$, their Jordan product is defined as
\begin{eqnarray*}
x\circ y:=(x^Ty,\ y_1\bar{x}+x_1\bar{y}).
\end{eqnarray*}
We write $x^2$ to mean $x\circ x$. The identity element is $e:=(1,0,\cdots,0)\in \mathbb{R}\times\mathbb{R}^{m-1}$. A pair of vectors $\{c_1,c_2\}$ is called a Jordan frame if it satisfies the following conditions:
\begin{eqnarray*}
\text{(i)}\ c_1\circ c_2=0;\ \ \text{(ii)}\ c_1^2=c_1\ \text{and} \ c_2^2=c_2;\ \ \text{(iii)}\ c_1+c_2=e.
\end{eqnarray*}
Every Jordan frame $\{c_1, c_2\}$ can be expressed in the form given in \eqref{decomposition}.

The definition of operator commute and its equivalent conditions are given as follows.

\begin{defi}\label{operator commute}{\rm \cite{F. Alizadeh and D. Goldfarb}}
Vectors $x$ and $y$ operator commute if they share a Jordan frame:
\[
x=\nu_{1}c_{1}+\nu_{2}c_{2},\ \ \ \  y=\upsilon_{1}c_{1}+\upsilon_{2}c_{2}
\]
for a Jordan frame $\{{c_{1}, c_{2}}\}$.
\end{defi}

\begin{lema}\label{equivalent}{\rm \cite{F. Alizadeh and D. Goldfarb}}
Vectors $x$ and $y$ operator commute if and only if $\bar{x}=0$, or $\bar{y}=0$, or $\bar{x}$ and $\bar{y}$ are proportional (i.e., $\bar{x}=\alpha \bar{y}$ for some real number $\alpha$).
\end{lema}
}

The following lemmas related to the second-order cone complementarity will be used.
\begin{lema}{\rm \cite{fukushima-luo-tseng2001}}
For any $x \in\mathcal{K}$ and $y \in\mathcal{K}$, we have
\begin{eqnarray*}
x^Ty=0  \ \Leftrightarrow\  x_i^Ty_i=0,\ i=1,\cdots,J.
\end{eqnarray*}
\end{lema}

\begin{lema}{\rm\cite{Liang Zhu Lin }}\label{cone complementarity-lemma}
For any $x \in{\rm bd}{\mathcal{K}}_i\backslash\{0\}$ and $y \in{\rm bd}{\mathcal{K}}_i\backslash\{0\}$, we have
\begin{eqnarray*}
x^Ty=0\ \Leftrightarrow \ x=k\hat{y}\ \ \mbox{with}\ k\in\mathbb{R}_{++}.
\end{eqnarray*}
\end{lema}

The following two properties related to projections onto ${\cal K}$ are very important.

\begin{lema}{\rm \cite{Andreani Haeser Schuverdt Silva }}\label{lema3}
{\rm (}Moreau's decomposition{\rm)} For every $\omega\in{\mathbb{R}^m}$, we have
    $$\omega=\Pi_{\cal K}(\omega)+\Pi_{{\cal K}^\circ}(\omega) ~~\text{and}~~ \langle\Pi_{\cal K}(\omega),\Pi_{{\cal K}^\circ}(\omega)\rangle=0.$$
\end{lema}

It follows from Lemma \ref{lema3} that for all $\omega\in{\mathbb{R}^m}$, we have
\begin{eqnarray*}
\langle \omega-\Pi_{\cal K}(\omega),\Pi_{\cal K}(\omega)\rangle=0.
\end{eqnarray*}

\begin{lema}{\rm \cite{Rockafellar}}\label{lema4}
$\upsilon = \Pi_{\mathcal{K}^\circ}(\omega)$ \ $\Leftrightarrow$ \
$\upsilon \in \mathcal{K}^\circ$, $\omega - \upsilon \in \mathcal{K}$, $\langle \omega - \upsilon, \upsilon \rangle = 0$.
\end{lema}

Given a feasible point $z^*$ of {SOCMPCC}, we define some useful index sets as follows
\begin{eqnarray*}\renewcommand{\arraystretch}{1.5}
\begin{array}{ll}
I_g(z^*):=\{i\ |\ g_i(z^*)=0\},&  C(z^*):=\{i\ |\ G_i(z^*)=l\hat{H}_i(z^*),\ l\in\mathbb{R_{++}}\},\\
I_G(z^*):=\{i\ |\ G_i(z^*)=0\},& I_G^+(z^*):=\{i\ |\ G_i(z^*)\in{{\rm int}\mathcal{K}_i}\},\\
I_H(z^*):=\{i\ |\ H_i(z^*)=0\},& I_H^+(z^*):=\{i\ |\ H_i(z^*)\in{{\rm int}\mathcal{K}_i}\},\\
B_G(z^*):=\{i\ |\ G_i(z^*)\in{{\rm bd}\mathcal {K}_i\backslash\{0\}}\},& B_H(z^*):=\{i\ |\ H_i(z^*)\in{{\rm bd}\mathcal {K}_i\backslash\{0\}}\}.
\end{array}
\end{eqnarray*}
We further define the following composite index sets:
\begin{eqnarray*}\renewcommand{\arraystretch}{1.5}
\begin{array}{ll}
I_1(z^*):=I_G(z^*)\cap I_H^+(z^*),& I_2(z^*):=I_H(z^*)\cap I_G^+(z^*),\\
I_3(z^*):=B_G(z^*)\cap B_H(z^*)\cap C(z^*),& I_4(z^*):=I_G(z^*)\cap B_H(z^*), \\
I_5(z^*):=B_G(z^*)\cap I_H(z^*),& I_6(z^*):=I_G(z^*)\cap I_H(z^*),\\
I(z^*):=I_1(z^*)\cup I_2(z^*)\cup\cdots\cup I_6(z^*),& I^c(z^*):=\{1,\cdots,J\}\backslash I(z^*),
\end{array}
\end{eqnarray*}
where for simplicity, we denote $I_k(z^*)$ by $I_k^*$ for $k=1,\cdots, 6$.

In the following, we introduce constraint qualifications and stationarity conditions for SOCMPCC that will be used throughout this paper. Specifically, using eigenvalue decomposition and cone slicing techniques for the second-order cone, Liang et al. \cite{Liang-Zhu-Lin 2026 JOTA} introduced a weak constraint qualification (SOCMPCC-weak-Robinson's CQ) for SOCMPCCs, which requires only the existence of a convergent subsequence at which the associated eigenvector derivatives are positively linearly independent.

\begin{defi} {\rm\cite{Liang-Zhu-Lin 2026 JOTA}}
Let $z^*$ be a feasible point of SOCMPCC. We say that $z^*$ satisfies {\rm SOCMPCC-weak-Robinson's constraint qualification (SOCMPCC-weak-Robinson's CQ)} if the following holds:
For every sequence $\{z^k\}_{k\in \mathbb{N}}\to z^*$, there exists an infinite subset $S\subseteq _{\infty} \mathbb{N}$ and convergent eigenvector sequences
\begin{eqnarray*}
&&\{u_{1}(G_{i}(z^k))\}_{k\in S}\to \tfrac{1}{2}(1, -\bar{w}_{i}), ~\{u_{2}(G_{i}(z^k))\}_{k\in S} \to \tfrac{1}{2}(1, \bar{w}_{i}),  \\
&&\{u_{1}(H_{i}(z^k))\}_{k\in S}\to \tfrac{1}{2}(1, -\bar{v}_{i}), ~\{u_{2}(H_{i}(z^k))\}_{k\in S}\to \tfrac{1}{2}(1, \bar{v}_{i}),
\end{eqnarray*}
with $\|\bar{w}_{i}\|=1$ for all $i\in I_G(z^*)$ and $\|\bar{v}_{i}\|=1$ for all $i\in I_H(z^*)$,
such that
\begin{eqnarray}\label{SOCMPCC-weak-nondegeneracy}
&&\{\nabla h_i(z^*)\}_{i\in \{1,\cdots,q\}} \nonumber\\
&&\cup\{\nabla g_i(z^*)\}_{i\in I_{g}(z^*)}  \nonumber\\
&&\cup \{\nabla G_{i}(z^*)u_{1}(G_{i}(z^*))\}_{i\in B_G(z^*)}\nonumber\\
&&\cup \{\nabla H_{i}(z^*)u_{1}(H_{i}(z^*))\}_{i\in B_H(z^*)}\\
&&\cup \{\nabla G_{i}(z^*)(1, -\bar{w}_{i}),\ \nabla G_{i}(z^*)(1, \bar{w}_{i})\}_{i\in I_G(z^*)}\nonumber\\
&&\cup \{\nabla H_{i}(z^*)(1, -\bar{v}_{i}),\ \nabla H_{i}(z^*)(1, \bar{v}_{i})\}_{i\in I_H(z^*)}\nonumber
\end{eqnarray}
is positively linearly independent.
\end{defi}

\begin{remark}\rm
As shown in \cite{Liang-Zhu-Lin 2026 JOTA}, SOCMPCC-weak-Robinson's CQ is weaker than SOCMPCC-weak-nondegeneracy condition, and both are weaker than SOCMPCC-nondegeneracy condition.
\end{remark}

\begin{defi}{\rm\cite{Ye Zhou}}\label{K-point}
We say that a feasible point $z^*$ of SOCMPCC is {\rm K-stationary} if there exist multipliers $(\lambda,\mu,\alpha,\beta)$ satisfying
\begin{eqnarray}
\nabla f(z^*)+\nabla g(z^*)\lambda+\nabla h(z^*)\mu+\sum\limits_{i=1}^J\nabla G_i(z^*)\alpha_i+\sum\limits_{i=1}^J\nabla H_i(z^*)\beta_i=0,
\end{eqnarray}
where
\begin{eqnarray}
\left\{\begin{array}{ll}
\lambda\geq 0,\ g(z^*)^T\lambda=0,\\
\beta_i=0,& i\in I_1(z^*),\\
\alpha_i=0,& i\in I_2(z^*),\\
\alpha_i\in\mathbb{R}\hat{G}_i(z^*),\ \beta_i\in\mathbb{R}\hat{H}_i(z^*),& i\in I_3(z^*),\\
\alpha_i\in {-\mathcal{K}_i}+\mathbb{R}_+H_i(z^*),\ \beta_i\in\mathbb{R}_-\hat{H}_i(z^*),& i\in I_4(z^*),\\
\alpha_i\in\mathbb{R}_-\hat{G}_i(z^*),\ \beta_i\in {-\mathcal{K}_i}+\mathbb{R}_+G_i(z^*),& i\in I_5(z^*),\\
\alpha_i\in{-\mathcal{K}_i},\ \beta_i\in{-\mathcal{K}_i},& i\in I_6(z^*).
\end{array}\right.
\end{eqnarray}
\end{defi}

\section{ALM Algorithm framework for SOCMPCC}

Note that for each $i = 1, \cdots, J$, if $G_i(z) \in \mathcal{K}_i$ and $H_i(z) \in \mathcal{K}_i$, then $G_i(z)^T H_i(z) \geq 0$ holds due to the self-duality of second-order cones.
Therefore, the SOCMPCC can be equivalently reformulated as an optimization problem with convex cone constraints, denoted by K-SOCMPCC:
\begin{eqnarray}\label{K-SOCMPCC}
\min   & & f(z) \nonumber\\
{\rm K\!-\!SOCMPCC}~~~~\qquad\qquad\mbox{s.t.} & & g(z)\le 0,  \ h(z)=0,   \nonumber\qquad\qquad\qquad\qquad\qquad\qquad\\
&& G(z)\in {\cal K},\ H(z)\in {\cal K},\\
&&G(z)^T H(z)\leq0.\nonumber
\end{eqnarray}

Given a {penalty} parameter $\rho > 0$, the augmented Lagrangian function associated with the SOCMPCC, denoted by $L_{\rho}$, is defined as
\begin{eqnarray*}
\begin{aligned}
L_{\rho}(z,\lambda,\mu,\eta,\Lambda,\Gamma) ~=~ &f(z)+\tfrac{\rho}{2}\vert\vert(g(z)+\tfrac{\lambda}{\rho})_+\vert\vert^2+\tfrac{\rho}{2}\vert\vert h(z)+\tfrac{\mu}{\rho}||^2\\
&+\tfrac{\rho}{2} \Big(\vert\vert\Pi_{{\cal K}^\circ}(G(z)+\tfrac{\Lambda}{\rho})||^2-||\tfrac{\Lambda}{\rho}||^2\Big)+\tfrac{\rho}{2} \Big(\vert\vert\Pi_{{\cal K}^\circ}(H(z)+\tfrac{\Gamma}{\rho})\vert\vert^2-\vert\vert\tfrac{\Gamma}{\rho}||^2\Big)\\
&+\tfrac{\rho}{2}\vert\vert(\langle G(z),H(z)\rangle +\tfrac{\eta}{\rho})_+\vert\vert^2,
\end{aligned}
\end{eqnarray*}
{where $\lambda$, $\mu$, $\eta$, $\Lambda$ and $\Gamma$ denote the Lagrange multipliers associated with the inequality, equality, complementarity, and cone constraints, respectively.}

The gradient of $L_{\rho}$ with respect to $z$ is given by
\begin{eqnarray*}\label{deriv}
\begin{aligned}
\nabla_z L_\rho(z,\lambda,\mu,\eta,\Lambda,\Gamma)~=~
&\nabla f(z)+\nabla g(z)(\lambda+\rho g(z))_++\nabla h(z)(\mu+\rho h(z))\nonumber\\
&+\nabla G(z)\big(\rho\Pi_{{\cal K}^\circ}(G(z)+\tfrac{\Lambda}{\rho})\big)+\nabla H(z)\big(\rho\Pi_{{\cal K}^\circ}(H(z)+\tfrac{\Gamma}{\rho})\big)\nonumber\\
&+\nabla (G(z)^TH(z))(\eta+\rho\langle G(z),H(z)\rangle )_+.
\end{aligned}
\end{eqnarray*}

\subsection{Algorithm framework and feasibility}
Based on the reformulation K-SOCMPCC, we now present a tailored algorithmic framework of the augmented Lagrangian method (ALM) for solving the SOCMPCC.

\begin{algorithm}[H]
\caption{ALM algorithm for SOCMPCC}
\label{alg:socmpcc}
\begin{algorithmic}[0] 

\STATE
Let $\{\epsilon_k\}$ be a nonincreasing sequence of positive real numbers converging to zero. {Let $\varepsilon$ be a termination precision.}
Let $B \subseteq \mathcal{K}^\circ$ be a nonempty, convex, and compact set. 
Choose parameters $\tau > 1$, $\sigma \in (0, 1)$, $\rho_0 > 0$, $\lambda_{\max} > 0$, $\mu_{\min} < \mu_{\max}$, $\eta_{\min} < \eta_{\max}$.
Initialize multipliers:
$\bar{\lambda}_i^0 \in [0, \lambda_{\max}]$ for $i = 1, \dots, p$,
$\bar{\mu}_i^0 \in [\mu_{\min}, \mu_{\max}]$ for $i = 1, \dots, q$,
$\bar{\eta}^0 \in [\eta_{\min}, \eta_{\max}]$,
and cone multipliers $\bar{\Lambda}^0, \bar{\Gamma}^0 \in B$.

\STATE \textbf{Step 1.}
For each $k > 0$, project $(\lambda^k, \mu^k, \eta^k)$ onto
$$
\bigotimes_{i=1}^p [0, \lambda_{\max}] \times \bigotimes_{i=1}^q [\mu_{\min}, \mu_{\max}] \times [\eta_{\min}, \eta_{\max}]
$$
to obtain $(\bar{\lambda}^k, \bar{\mu}^k, \bar{\eta}^k)$.
Compute $z^k$ such that
$$
\|\nabla_z L_{\rho}(z^k,\bar{\lambda}^k,\bar{\mu}^k,\bar{\eta}^k,\bar{\Lambda}^k,\bar{\Gamma}^k)\| \le \epsilon_k.
$$

\STATE \textbf{Step 2.} { For $k=1,2,\dots$, if $\|z^k-z^{k-1}\|\leq \varepsilon$, then stop.} Otherwise, update multipliers:
\begin{eqnarray*}
&&\lambda^{k+1}=(\bar{\lambda}^k+\rho_kg(z^k))_+,~~\mu^{k+1}=\bar{\mu}^k+\rho_kh(z^k),~~\eta^{k+1}=(\bar{\eta}^k+\rho_k\langle G(z^k),H(z^k)\rangle)_+,\\
&&{{\Lambda}}^k=\rho_k\Pi_{{\cal K}^\circ}(G(z^k)+\tfrac{\bar{\Lambda}^k}{\rho_k}),~~{{\Gamma}}^k=\rho_k\Pi_{{\cal K}^\circ}(H(z^k)+\tfrac{\bar{\Gamma}^k}{\rho_k}),~~ {\bar{\Lambda}}^{k+1}=\Pi_B(\Lambda^k),\\
&&{\bar{\Gamma}}^{k+1}=\Pi_B(\Gamma^k),~~\zeta_{k+1}=\min\{\lambda^{k+1},-g(z^k)\},~~\delta_{k+1}=\min\{\eta^{k+1},-\langle G(z^k),H(z^k)\rangle\},\\
&&V^k=\tfrac{\bar{\Lambda}^k}{\rho_k}-\Pi_{{\cal K}^\circ}(G(z^k)+\tfrac{\bar{\Lambda}^k}{\rho_k}),~~
U^k=\tfrac{\bar{\Gamma}^k}{\rho_k}-\Pi_{{\cal K}^\circ}(H(z^k)+\tfrac{\bar{\Gamma}^k}{\rho_k}).
\end{eqnarray*}

\STATE \textbf{Step 3.}
If $k = 0$ or the following two conditions hold:
\begin{eqnarray}
  &&\max\{||\zeta_{k+1}||_{\infty},|\delta_{k+1}|,||h(z^k)||_{\infty}\}\le\sigma \max\{||\zeta_{k}||_{\infty},|\delta_{k}|,||h(z^{k-1})||_{\infty}\}, \label{2.3}\\
&&||V^k||\le\sigma||V^{k-1}||, ~~||U^k||\le\sigma||U^{k-1}||,\label{2.4}
\end{eqnarray}
set $\rho_{k+1} = \rho_k$. Otherwise, update the penalty parameter as $\rho_{k+1} = \tau \rho_k$, and return to {\bf Step 1}.

\end{algorithmic}
\end{algorithm}

Theorem \ref{feasibility-thm} establishes the feasibility properties for the proposed Algorithm \ref{alg:socmpcc}.

\begin{thm}\label{feasibility-thm}
Let $\{z^k\}$ be a sequence generated by {\rm Algorithm \ref{alg:socmpcc}}. Suppose that $z^*$ is a limit point of $\{z^k\}$.
If the penalty parameter sequence $\{\rho_k\}$ is bounded, or there exists a constant $\Upsilon>0$ such that
\begin{eqnarray}\label{2.6}
L_{\rho_k}(z^k,{\bar{\lambda}}^k,{\bar{\mu}}^k,{\bar{\eta}}^k,{\bar{\Lambda}}^k,{\bar{\Gamma}}^k)\le\Upsilon,&&\forall ~k,
\end{eqnarray}
then $z^*$ is feasible to the SOCMPCC.
\end{thm}
\begin{proof}
By taking a subsequence if necessary, we may assume $z^k \to z^*$. We analyze the feasibility by considering two separate cases based on whether the sequence $\{\rho_k\}$ is bounded.

Case 1: The sequence $\{\rho_k\}$ is bounded.
Then, the penalty parameters are not updated beyond some iteration index $k_0 > 0$, and we may assume $\rho_k \to \rho^*$. As a result, the update condition in (\ref{2.3}) holds for all $k > k_0$. From the definitions of $\zeta_{k+1}$ and $\delta_{k+1}$ in Step 2 of Algorithm \ref{alg:socmpcc}, we obtain
$$
\lim_{k \to \infty} h(z^k) = 0, \quad
\lim_{k \to \infty} g(z^k) \le -\lim_{k \to \infty} \zeta_{k+1} = 0, \quad
\lim_{k \to \infty} \langle G(z^k), H(z^k) \rangle \le -\lim_{k \to \infty} \delta_{k+1} = 0.
$$
Taking the limit as $k \to \infty$, we have
$$
h(z^*) = 0, \quad g(z^*) \le 0, \quad \langle G(z^*), H(z^*) \rangle \le 0.
$$

Next, from \eqref{2.4} and the definition of $V^k$, we have
$$
\lim_{k \to \infty} \left( \tfrac{\bar{\Lambda}^k}{\rho_k} - \Pi_{\mathcal{K}^\circ}\left( G(z^k) + \tfrac{\bar{\Lambda}^k}{\rho_k} \right) \right) = 0.
$$
Since $\bar{\Lambda}^k \in B$ and $B$ is a convex and compact set, the sequence $\left\{ \tfrac{\bar{\Lambda}^k}{\rho_k} \right\}$ is bounded. Using the nonexpansiveness of the projection operator, we deduce that both sequences
$$
\tfrac{\bar{\Lambda}^k}{\rho_k} \quad \text{and} \quad \Pi_{\mathcal{K}^\circ} \left( G(z^k) + \tfrac{\bar{\Lambda}^k}{\rho_k} \right)
$$
have convergent subsequences. Denote the limits by
$$
\tfrac{\bar{\Lambda}^k}{\rho_k} \to \tfrac{\bar{\Lambda}^*}{\rho^*}, \quad \Pi_{\mathcal{K}^\circ} \left( G(z^k) + \tfrac{\bar{\Lambda}^k}{\rho_k} \right) \to \Pi_{\mathcal{K}^\circ} \left( G(z^*) + \tfrac{\bar{\Lambda}^*}{\rho^*} \right).
$$
Therefore,
$$
\tfrac{\bar{\Lambda}^*}{\rho^*} = \Pi_{\mathcal{K}^\circ} \left( G(z^*) + \tfrac{\bar{\Lambda}^*}{\rho^*} \right),
$$
which implies $\tfrac{\bar{\Lambda}^*}{\rho^*} \in \mathcal{K}^\circ$.

By Lemma \ref{lema3}, we have
$$
\left\langle G(z^*) + \tfrac{\bar{\Lambda}^*}{\rho^*} - \Pi_{\mathcal{K}^\circ} \left( G(z^*) + \tfrac{\bar{\Lambda}^*}{\rho^*} \right), \Pi_{\mathcal{K}^\circ} \left( G(z^*) + \tfrac{\bar{\Lambda}^*}{\rho^*} \right) \right\rangle = 0.
$$
Substituting the identity above, we conclude
$$
\langle G(z^*), \tfrac{\bar{\Lambda}^*}{\rho^*} \rangle = 0.
$$
Then, by Lemma \ref{lema4}, it follows that $G(z^*) \in \mathcal{K}$. By analogous reasoning, we also have $H(z^*) \in \mathcal{K}$. Hence, $z^*$ is feasible to SOCMPCC.

Case 2: The sequence $\{\rho_k\}$ is unbounded.
Since $\rho_k$ is nondecreasing, we must have $\rho_k \to \infty$. From condition (\ref{2.6}), we have
\begin{eqnarray*}
    \tfrac{2(\Upsilon-f(z^k))}{\rho_k} &\geq &||\Pi_{{\cal K}^\circ}(G(z^k)+\tfrac{\bar{\Lambda}^k}{\rho_k})||^2-\vert\vert\tfrac{\bar{\Lambda}^k}{\rho_k}\vert\vert^2 +{   ||\Pi_{{\cal K}^\circ}(H(z^k)+\tfrac{\bar{\Gamma}^k}{\rho_k})||^2}-\vert\vert\tfrac{\bar{\Gamma}^k}{\rho_k}\vert\vert^2\\
    &&+\vert\vert(g(z^k)+\tfrac{\bar{\lambda}^k}{\rho_k})_+\vert\vert^2+\vert\vert h(z^k)+\tfrac{\bar{\mu}^k}{\rho_k}\vert\vert^2 + \vert\vert({   \langle G(z^k),H(z^k)\rangle}+\tfrac{\bar{\eta}^k}{\rho_k})_+\vert\vert^2.
     \end{eqnarray*}
Passing to the limit along a convergent subsequence and using $\rho_k \to \infty$, each term on the left-hand side tends to zero, which implies the feasibility of $z^*$.

Thus, in both cases, $z^*$ is feasible to SOCMPCC.
This completes the proof.
\end{proof}

\begin{remark}\rm 
The proof of Theorem \ref{feasibility-thm} indicates that, in addition to the boundedness of the penalty parameter, the boundedness of the augmented Lagrangian function (i.e., condition \eqref{2.6}) also guarantees the feasibility of the limit point. Therefore, if \eqref{2.6} holds, then $z^*$ is feasible to the SOCMPCC.
\end{remark}

\subsection{Convergence analysis}

Consider the following unconstrained optimization problem:

\begin{eqnarray}\label{2.8}
\underset{}{\min} ~~ \tfrac{1}{2}\Big(||g(z)_+||^2+||h(z)||^2+||\langle G(z),H(z)\rangle_+||^2+||\Pi_{{\cal K}^\circ}(G(z))||^2+||\Pi_{{\cal K}^\circ}(H(z))||^2\Big).
\end{eqnarray}

Theorem \ref{thm2} shows that when the penalty parameter sequence is unbounded, any limit point of the sequence generated by the proposed ALM is a stationary point of problem \eqref{2.8}.

\begin{thm}\label{thm2}
If the penalty parameter sequence $\{\rho_k\}$ is unbounded, then any limit point of a sequence $\{z^k\}$ generated by {\rm Algorithm \ref{alg:socmpcc}} is a stationary point of \eqref{2.8}.
\end{thm}
\begin{proof}
The gradient of the objective function in \eqref{2.8} is
\begin{eqnarray}\label{2.7}
\begin{aligned}
&\nabla g(z)(g(z))_++\nabla h(z)h(z)+\nabla (G(z)^TH(z))(\langle G(z),H(z)\rangle)_+\\
&+{ \nabla G(z)\Pi_{{\cal K}^\circ}(G(z))+\nabla H(z)\Pi_{{\cal K}^\circ}(H(z)).}
\end{aligned}
\end{eqnarray}

Assume, without loss of generality, that $z^k\to z^*$ along a subsequence.
Since $\{\rho_k\}$ is unbounded, we have $\rho_k\to\infty$ as $k\to\infty$.
By Step 1 of Algorithm \ref{alg:socmpcc} and \eqref{deriv}, for each \(k\), there exists \(\pi^k\) such that \(\|\pi^k\| \leq \epsilon_k\) and
\begin{eqnarray}\label{2.9}
\begin{aligned}
\pi^k ~=~ &\nabla f(z^k)+\nabla g(z^k)({\bar{\lambda}}^k+{\rho}_k g(z^k))_++\nabla h(z^k)({\bar{\mu}}^k+{\rho}_k h(z^k))\\
&+{ \nabla G(z^k)({\rho}_k\Pi_{{\cal K}^\circ}(G(z^k)+\tfrac{{\bar{\Lambda}}^k}{{\rho}_k}))+\nabla H(z^k)({\rho}_k\Pi_{{\cal K}^\circ}(H(z^k)+\tfrac{{\bar{\Gamma}}^k}{{\rho}_k}))}\\
&+\nabla (G(z^k)^TH(z^k))({\bar{\eta}}^k+{\rho}_k\langle G(z^k),H(z^k)\rangle )_+.
\end{aligned}
\end{eqnarray}

Since \(\epsilon_k \to 0\), it follows that \(\pi^k \to 0\) as \(k \to \infty\). Dividing \eqref{2.9} by \(\rho_k\) yields
\begin{eqnarray*}
\begin{aligned}
\tfrac{\pi^k}{{\rho}_k} ~=~ &\tfrac{\nabla f(z^k)}{{\rho}_k}+\nabla g(z^k)(\tfrac{{\bar{\lambda}}^k}{{\rho}_k}+ g(z^k))_++\nabla h(z^k)(\tfrac{{\bar{\mu}}^k}{{\rho}_k}+ h(z^k))\\
&+{ \nabla G(z^k)\Pi_{{\cal K}^\circ}(G(z^k)+\tfrac{{\bar{\Lambda}}^k}{{\rho}_k})+\nabla H(z^k)\Pi_{{\cal K}^\circ}(H(z^k)+\tfrac{{\bar{\Gamma}}^k}{{\rho}_k})}\\
&+\nabla (G(z^k)^TH(z^k))(\tfrac{{\bar{\eta}}^k}{{\rho}_k}+\langle G(z^k),H(z^k)\rangle )_+,
\end{aligned}
\end{eqnarray*}
Taking the limit as \(k \to \infty\), we obtain
\begin{eqnarray*}
\begin{aligned}
0~=~ &\nabla g(z^*)(g(z^*))_++\nabla h(z^*)h(z^*)
+\nabla (G(z^*)^TH(z^*))(\langle G(z^*),H(z^*)\rangle)_+\\
&+{ \nabla G(z^*)\Pi_{{\cal K}^\circ}(G(z^*))+\nabla H(z^*)\Pi_{{\cal K}^\circ}(H(z^*)).}
\end{aligned}
\end{eqnarray*}
This confirms that \(z^*\) is indeed a stationary point of \eqref{2.8}. This completes the proof.
\end{proof}

Ye and Zhou \cite{Ye Zhou} demonstrated that the K-stationary condition of SOCMPCC is equivalent to the KKT condition of (\ref{K-SOCMPCC}), namely, there exists $(\lambda,\mu,a,b,\gamma)\in \mathbb{R}^p\times\mathbb{R}^q\times\mathbb{R}^m\times\mathbb{R}^m\times\mathbb{R}$ such that 
\begin{eqnarray*}
\left\{
\begin{array}{ll}
\nabla f(z^*)+\nabla g(z^*)\lambda+\nabla h(z^*)\mu+\sum\limits_{i=1}^J \nabla G_i(z^*)a_i+\sum\limits_{i=1}^J \nabla H_i(z^*)b_i+\gamma\nabla(G(z^*)^T H(z^*))=0,\\
\lambda \geq 0,\ g(z^*)^T\lambda = 0,\\
G_i(z^*)\in \mathcal{K}_i,\ -a_i\in \mathcal{K}_i,\ G_i(z^*)^T a_i = 0,\  i=1,\dots,J,\\
H_i(z^*)\in \mathcal{K}_i,\ -b_i\in \mathcal{K}_i,\ H_i(z^*)^T b_i = 0,\  i=1,\dots,J,\\
\gamma\geq 0.
\end{array}
\right.
\end{eqnarray*}

Having characterized the KKT condition above, we now prove the convergence of Algorithm \ref{alg:socmpcc}. Specifically, we show that any limit point of the sequence generated by Algorithm \ref{alg:socmpcc} satisfies the KKT condition, and hence is a K-stationary point.

Theorem \ref{convergence-thm1} obtains that when the penalty parameter sequence is bounded, any limit point of \(\{z^k\}\) generated by Algorithm \ref{alg:socmpcc} is not only feasible for the SOCMPCC, but also K-stationary.

\begin{thm}\label{convergence-thm1}
Let \(\{z^k\}\) be a sequence generated by {\rm Algorithm \ref{alg:socmpcc}}. Suppose that \(z^*\) be a limit point of \(\{z^k\}\).
If the penalty parameter sequence $\{\rho_k\}$ is bounded, then \(z^*\) is K-stationary for the SOCMPCC.
\end{thm}
\begin{proof}
Since $\{\rho_k\}$ is bounded, there exists $k_0>0$ such that for all $k>k_0$, the parameter $\rho_{k+1}$ ceases to update.
Without loss of generality, by considering a subsequence if necessary, we may assume $\rho_{k+1}\to \rho^*$.
From Step 2 of Algorithm \ref{alg:socmpcc}, the definitions of $\lambda^{k+1}$, $\mu^{k+1}$, $\eta^{k+1}$, $\Lambda^{k}$, and $\Gamma^{k}$ ensure that these sequences are bounded. Thus, we may assume the following limits exist:
\[
\lambda^{k+1}\to \lambda^*, \quad \mu^{k+1}\to \mu^*, \quad \eta^{k+1}\to \eta^*, \quad \Lambda^{k}\to \Lambda^*, \quad \Gamma^{k}\to \Gamma^*.
\]

By the definition of ${\Lambda}^k$ and Lemma \ref{lema3}, $\langle{\Lambda}^k,\Pi_{{\cal K}}(\rho_k(G(z^k)+\bar{\Lambda}^k))\rangle=0$ holds for all $k$, which yields
\begin{align*}
\langle{\Lambda}^k,\Pi_{\cal K}(G(z^k))\rangle
= \tfrac{1}{\rho_k}\langle{\Lambda}^k,\Pi_{\cal K}(\rho_kG(z^k))\rangle
= \tfrac{1}{\rho_k}\langle{\Lambda}^k,\Pi_{\cal K}(\rho_kG(z^k))-\Pi_{\cal K}(\rho_kG(z^k)+\bar{\Lambda}^k)\rangle.
\end{align*}

Using the nonexpansiveness of the projection and Lemma \ref{lema3}, we obtain
\begin{align*}
\|\Pi_{\cal K}(\rho_kG(z^k))-\Pi_{\cal K}(\rho_kG(z^k)+\bar{\Lambda}^k)\|
~\leq~& \|\rho_kG(z^k)-\Pi_{\cal K}(\rho_kG(z^k)+\bar{\Lambda}^k)\| \\
~=~& \|\rho_kG(z^k)+\bar{\Lambda}^k-\Pi_{\cal K}(\rho_k(G(z^k)+\bar{\Lambda}^k))-\bar{\Lambda}^k\| \\
~=~& \|\Pi_{{\cal K}^\circ}(\rho_k(G(z^k)+\bar{\Lambda}^k))-\bar{\Lambda}^k\| \\
~=~& \|{\Lambda}^k-\bar{\Lambda}^k\|.
\end{align*}

Since ${\Lambda}^k-\bar{\Lambda}^k=\Pi_{{\cal K}^\circ}(\rho_kG(z^k)+\bar{\Lambda}^k)-\bar{\Lambda}^k=-\rho_kV^k\to 0$ and ${\Lambda}^k=\bar{\Lambda}^k-\rho_kV^k$ forms a bounded sequence, it follows that $\langle{\Lambda}^k,\Pi_{\cal K}(G(z^k))\rangle\to 0$. Consequently, $\langle{\Lambda}^*,G(z^*)\rangle=0$.
Similarly, we have $\langle{\Gamma}^*,H(z^*)\rangle=0$.
From the definitions of ${\Lambda}^k$ and ${\Gamma}^k$ in Algorithm \ref{alg:socmpcc}, we conclude ${\Lambda}^*,{\Gamma}^*\in{\cal K}^\circ$.

The point $z^*$ is feasible to SOCMPCC problem, meaning $g(z^*)\leq 0$.
From the definition of $\zeta^{k+1}$ in Algorithm \ref{alg:socmpcc} and the multiplier update rule, when $\{\rho_k\}$ is bounded, $\lambda^*=0$ whenever $g(z^*)<0$.
Thus, the complementarity condition $g(z^*)^T\lambda^*=0$ holds.

In summary, there exist multipliers $(\lambda^*,\mu^*,\Lambda^*,\Gamma^*,\eta^*)$ satisfying the following KKT condition:
\begin{equation}\label{KKT}
\begin{cases}
\nabla f(z^*)+\nabla g(z^*)\lambda^*+\nabla h(z^*)\mu^*+\sum\limits_{i=1}^J\nabla G_i(z^*)\Lambda_i^*+\sum\limits_{i=1}^J\nabla H_i(z^*)\Gamma_i^*+\eta^*\nabla(G(z^*)^TH(z^*))=0 \\
\lambda^*\geq 0,\ g(z^*)^T\lambda^*=0, \\
G_i(z^*)\in {\cal K}_i, -\Lambda^*_i\in {\cal K}_i, G_i(z^*)^T\Lambda^*_i=0, \quad i=1,\cdots,J, \\
H_i(z^*)\in {\cal K}_i, -\Gamma^*_i\in {\cal K}_i, H_i(z^*)^T\Gamma^*_i=0, \quad i=1,\cdots,J, \\
\eta^*\geq 0.
\end{cases}
\end{equation}
Therefore, when $\{\rho_k\}$ is bounded, $z^*$ is a K-stationary point of SOCMPCC.
\end{proof}

{
\begin{thm}\label{convergence-thm2}
Let \(\{z^k\}\) be a sequence generated by {\rm Algorithm \ref{alg:socmpcc}}. Suppose that \(z^*\) is a limit point of \(\{z^k\}\) and it is feasible to SOCMPCC.
If the penalty parameter sequence $\{\rho_k\}$ is unbounded and the SOCMPCC-weak-Robinson's CQ holds at \(z^*\), then \(z^*\) is K-stationary for the SOCMPCC.
\end{thm}
\begin{proof}
As in Theorem \ref{feasibility-thm}, without loss of generality, assume that $z^k\to z^*$. From Algorithm \ref{alg:socmpcc}, there exists a sequence $\{\pi^k\}$ with $\|\pi^k\| \leq \epsilon_k$ such that
\begin{eqnarray}
\begin{aligned}\label{Equation-1}
\pi^k~=~&\nabla f(z^k) + \nabla g(z^k)(\bar{\lambda}^k + \rho_k g(z^k))_+ + \nabla h(z^k)(\bar{\mu}^k + \rho_k h(z^k)) \\
&+{  \nabla G(z^k)(\rho_k\Pi_{\mathcal{K}^\circ}(G(z^k) + \tfrac{\bar{\Lambda}^k}{\rho_k}))
+ \nabla H(z^k)(\rho_k\Pi_{\mathcal{K}^\circ}(H(z^k) + \tfrac{\bar{\Gamma}^k}{\rho_k})) }\\
&+ \nabla (G(z^k)^TH(z^k))(\bar{\eta}^k + \rho_k\langle G(z^k), H(z^k)\rangle)_+.
\end{aligned}
\end{eqnarray}
This can be rewritten as
\begin{eqnarray*}
\pi^k~=~\nabla f(z^k) + \nabla g(z^k)\lambda^{k+1} + \nabla h(z^k)\mu^{k+1} +{  \nabla G(z^k)\Lambda^{k} + \nabla H(z^k)\Gamma^{k}}
+ \eta^{k+1}\nabla (G(z^k)^TH(z^k)).
\end{eqnarray*}

For any $i\notin I_g(z^*)$, the continuity of $g$ at $z^*$ ensures the existence of $\xi_1 < 0$ such that $g_i(z^k) < \xi_1$ for all sufficiently large $k$.
From the definition of $\lambda^{k+1}$ in Step 2 of Algorithm \ref{alg:socmpcc} and the boundedness of $\{\bar{\lambda}^k\}$, it follows that $\lambda_i^{k+1} = 0$ for $i\notin I_g(z^*)$ when $k$ is sufficiently large.
Similarly, for $i\in I_{G}^+(z^*)$ (i.e., $G_i(z^*)\in \operatorname{int}\mathcal{K}_i$), we have $G_i(z^k)\in \operatorname{int}\mathcal{K}_i$ for large $k$.
Consequently, $\Lambda_i^{k}=0$ for such indices. An analogous argument shows that $\Gamma_i^{k}=0$ {for $i\in I_{H}^+(z^*)$} for sufficiently large $k$.

Therefore, for all sufficiently large $k$, it holds
\begin{align}\label{Equation-2}
\pi^k ~=~&\nabla f(z^k) + \sum\limits_{i\in I_g(z^*)}\lambda_i^{k+1}\nabla g_i(z^k) + \sum\limits_{i=1}^q\mu_i^{k+1}\nabla h_i(z^k) + \sum\limits_{i=1}^J\eta^{k+1}\nabla (G_i(z^k)^TH_i(z^k))\nonumber\\
&+ \sum\limits_{i\in I_G(z^*)\cup B_G(z^*)}\nabla G_i(z^k)\Lambda_i^k
+ \sum\limits_{i\in I_H(z^*)\cup B_H(z^*)}\nabla H_i(z^k)\Gamma_i^k.
\end{align}

Using the differential rule, we have
\begin{align*}
\sum\limits_{i=1}^J\eta^{k+1}\nabla (G_i(z^k)^TH_i(z^k))~=~&\sum\limits_{i=1}^J\eta^{k+1}\nabla G_i(z^k)H_i(z^k) + \sum\limits_{i=1}^J\eta^{k+1}\nabla H_i(z^k)G_i(z^k) \\
~=~&\sum_{i\in I_1^*\cup I_3^*\cup I_4^*}\eta^{k+1} \nabla G_i(z^k) H_i(z^k)\\
&+\sum_{i\in I_2^*\cup I_3^*\cup I_5^*}\eta^{k+1}\nabla H_i(z^k) G_i(z^k)+A^k,
\end{align*}
where 
\begin{align*}
A^k~=~&\sum_{i\in I_2^*}\eta^{k+1}\nabla G_i(z^k) H_i(z^k)+\sum_{i\in I_1^*}\eta^{k+1}\nabla H_i(z^k) G_i(z^k)+\sum_{i\in I_5^*}\eta^{k+1}\nabla G_i(z^k) H_i(z^k)\\
&+\sum_{i\in I_4^*}\eta^{k+1}\nabla H_i(z^k) G_i(z^k)+\sum_{i\in I_6^*}\eta^{k+1}\nabla G_i(z^k) H_i(z^k)+\sum_{i\in I_6^*}\eta^{k+1}\nabla H_i(z^k) G_i(z^k).
\end{align*}
Since $H_i(z^k)\to0$ for $i\in I_2^*\cup I_5^*\cup I_6^*$, and $G_i(z^k)\to 0$ for $i\in I_1^*\cup I_4^*\cup I_6^*$, it follows that, when $\eta^{k+1}$ is bounded, $A^k\to 0$.

We define $\tilde{\Lambda}^k=\Pi_{\mathcal{K}^\circ}(\rho_kG(z^k))$. Using the nonexpansiveness of the projection, we obtain
\begin{align*}
\|\Lambda^k-\tilde{\Lambda}^k\|\leq \|\bar{\Lambda}^k\|.
\end{align*}
Since $\bar{\Lambda}^k \in B$ and $B$ is a convex and compact set, $\{\Lambda^k-\tilde{\Lambda}^k\}$ is bounded. 
Similarly, defining $\tilde{\Gamma}^k=\Pi_{\cal K^\circ}(\rho_kH(z^k))$, we have that $\{\Gamma^k-\tilde{\Gamma}^k\}$ is bounded.

So, \eqref{Equation-2} can be equivalently stated as
\begin{align}\label{Equation-3}
\pi^k ~=~& \nabla f(z^k) + \sum\limits_{i\in I_g(z^*)}\lambda_i^{k+1}\nabla g_i(z^k) + \sum\limits_{i=1}^q\mu_i^{k+1}\nabla h_i(z^k) + \sum\limits_{i=1}^J\nabla G_i(z^k)\tilde\Lambda_i^k+ \sum\limits_{i=1}^J\nabla H_i(z^k)\tilde\Gamma_i^k\nonumber\\
&+\sum_{i\in I_1^*\cup I_3^*\cup I_4^*}\eta^{k+1} \nabla G_i(z^k) H_i(z^k)+\sum_{i\in I_2^*\cup I_3^*\cup I_5^*}\eta^{k+1}\nabla H_i(z^k) G_i(z^k)\nonumber\\
&+\sum_{i=1}^{J}\nabla G_i(z^k)(\Lambda_i^k-\tilde{\Lambda}_i^k)+\sum_{i=1}^{J}\nabla H_i(z^k)(\Gamma_i^k-\tilde{\Gamma}_i^k)+A^k.
\end{align}

For notational convenience, we write
$\lambda_{1}^{H_{i}, k}$, $u_{1}^{H_{i}, k}$, $\lambda_{2}^{H_{i}, k}$, $u_{2}^{H_{i}, k}$, $\lambda_{1}^{G_{i}, k}$, $u_{1}^{G_{i}, k}$, $\lambda_{2}^{G_{i}, k}$, and $u_{2}^{G_{i}, k}$ 
for
$\lambda_1(H_i(z^k))$, $u_1(H_i(z^k))$, $\lambda_2(H_i(z^k))$, $u_2(H_i(z^k))$, $\lambda_1(G_i(z^k))$, $u_1(G_i(z^k))$, $\lambda_2(G_i(z^k))$, and $u_2(G_i(z^k))$,
respectively.
We next consider the spectral decompositions
\begin{eqnarray*}
&&\tilde{\Lambda}_i^k=\Pi_{\mathcal{K}_i^\circ}(\rho_kG_i(z^k))
           = -(\xi_i^k u_1^{G_{i},k} + \phi_i^k u_2^{G_{i},k}),\\
&&\tilde{\Gamma}_i^k=\Pi_{\mathcal{K}_i^\circ}(\rho_kH_i(z^k))
           =-(\zeta_i^k u_1^{H_{i},k} + \delta_i^k u_2^{H_{i},k}),
\end{eqnarray*}
where $\xi_i^k=(-\rho_k\lambda_1^{G_{i},k})_{+}\geq 0$, $\phi_i^k=(-\rho_k\lambda_2^{G_{i},k})_{+}\geq 0$, $\zeta_i^k=(-\rho_k\lambda_1^{H_{i},k})_{+}\geq 0$ and $\delta_i^k=(-\rho_k\lambda_2^{H_{i},k})_{+}\geq 0$.
It should be noted that when $i \in I_{3}^*\cup I_{5}^*$, we have $\lambda_{2}^{G_{i},k} > 0$; when $i \in I_{3}^*\cup I_{4}^*$, we have $\lambda_{2}^{H_{i},k}> 0$.
According to the definitions of $\phi_{i}^k$ and $\delta_{i}^k$, we obtain $\phi_{i}^k=0$ for $i \in I_{3}^*\cup I_{5}^*$ and $\delta_{i}^k=0$ for $i \in I_{3}^*\cup I_{4}^*$.
Using the properties of the index sets, one can verify that
$$
\sum\limits_{i\in I_3^*\cup I_5^*}\nabla G_i(z^k)\tilde{\Lambda}_i^k+\sum\limits_{i\in I_3^*\cup I_4^*}\nabla H_i(z^k)\tilde{\Gamma}_i^k=-\sum\limits_{i\in I_3^*\cup I_5^*}\xi_{i}^{k}\nabla G_i(z^k)u_{1}^{G_{i},k}-\sum\limits_{i\in I_3^*\cup I_4^*}\zeta_{i}^{k}\nabla H_i(z^k)u_{1}^{H_{i},k}.
$$
When $i\in I_2^*$, we have $G_{i}(z^k)\in {\rm int}{\cal K}_{i}$. Therefore, $\tilde{\Lambda}_i^k=\Pi_{\mathcal{K}_i^\circ}(\rho_kG_i(z^k))=0$.
Similarly, when $i\in I_1^*$, we have $\tilde{\Gamma}_i^k=\Pi_{\mathcal{K}_i^\circ}(\rho_kH_i(z^k))=0$.
Therefore, $\sum\limits_{i\in I_2^*}\nabla G_i(z^k)\tilde{\Lambda}_i^k+\sum\limits_{i\in I_1^*}\nabla H_i(z^k)\tilde{\Gamma}_i^k=0$.

Using (\ref{decomposition}), we obtain
\begin{eqnarray*}
&&H_i(z^k)=\lambda_1^{H_{i},k}u_1^{H_{i},k}+\lambda_2^{H_{i},k}u_2^{H_{i},k}, \quad i\in I_1^*\cup I_3^*\cup I_4^*,\\
&&G_i(z^k)=\lambda_1^{G_{i},k}u_1^{G_{i},k}+\lambda_2^{G_{i},k}u_2^{G_{i},k}, \quad i\in I_2^*\cup I_3^*\cup I_5^*.
\end{eqnarray*}

Therefore, we can rewrite \eqref{Equation-3} as
\begin{align}\label{Equation-4}
\pi^k~=~&\nabla f(z^k)+\sum_{i\in I_g(z^*)}\lambda_i^{k+1}\nabla g_i(z^k) + \sum_{i=1}^q\mu_i^{k+1}\nabla h_i(z^k)\nonumber \\
&-\sum_{i\in I_1^*\cup I_4^*\cup I_6^*}\nabla G_i(z^k)(\xi_{i}^{k}u_{1}^{G_{i}, k}+\phi_{i}^{k}u_{2}^{G_{i}, k})-\sum_{i\in I_3^*\cup I_5^*}\xi_{i}^{k}\nabla G_i(z^k)u_{1}^{G_{i}, k}\nonumber \\
&-\sum_{i\in I_2^*\cup I_5^*\cup I_6^*}\nabla H_i(z^k)(\zeta_{i}^{k}u_{1}^{H_{i}, k}+\delta_{i}^{k}u_{2}^{H_{i}, k})-\sum_{i\in I_3^*\cup I_4^*}\zeta_{i}^{k}\nabla H_i(z^k)u_{1}^{H_{i}, k}\nonumber\\
&+\sum_{i\in I_1^*}\eta^{k+1}\nabla G_i(z^k)(\lambda_{1}^{H_{i}, k}u_{1}^{H_{i}, k}+\lambda_{2}^{H_{i}, k}u_{2}^{H_{i}, k})+\sum_{i\in I_3^*\cup I_4^*}\eta^{k+1}\lambda_{2}^{H_{i}, k}\nabla G_i(z^k)u_{2}^{H_{i}, k}\nonumber\\
&+\sum_{i\in I_2^*}\eta^{k+1}\nabla H_i(z^k)(\lambda_{1}^{G_{i}, k}u_{1}^{G_{i}, k}+\lambda_{2}^{G_{i}, k}u_{2}^{G_{i}, k})+\sum_{i\in I_3^*\cup I_5^*}\eta^{k+1}\lambda_{2}^{G_{i}, k}\nabla H_i(z^k)u_{2}^{G_{i}, k}\nonumber\\
&+\sum_{i=1}^{J}\nabla G_i(z^k)(\Lambda_i^k-\tilde{\Lambda}_i^k)+\sum_{i=1}^{J}\nabla H_i(z^k)(\Gamma_i^k-\tilde{\Gamma}_i^k)+A^k+B^k\to 0,
\end{align}
where
\begin{align*}
B^k~=~\sum_{i\in I_3^*\cup I_4^*}\eta^{k+1}\lambda_{1}^{H_{i}, k}\nabla G_i(z^k)u_{1}^{H_{i}, k}+\sum_{i\in I_3^*\cup I_5^*}\eta^{k+1}\lambda_{1}^{G_{i}, k}\nabla H_i(z^k)u_{1}^{G_{i}, k}.
\end{align*}
Since $\lambda_{1}^{H_{i}, k}\to0$ for all $i\in I_3^*\cup I_4^*$ and $\lambda_{1}^{G_{i}, k}\to0$ for all $I_3^*\cup I_5^*$ as $k\to \infty$, when $\eta^{k+1}$ is bounded, we obtain $B^k\to 0$.

Let $k\to \infty$. Then $\lambda_{2}^{H_{i}, k}\to \lambda_{2}^{H_{i}}\neq0$ for $i\in I_1^*\cup I_3^*\cup I_4^*$, $\lambda_{1}^{H_{i}, k}\to \lambda_{1}^{H_{i}}\neq0$ for $i\in I_1^*$, $\lambda_{2}^{G_{i}, k}\to \lambda_{2}^{G_{i}}\neq0$ for $i\in I_2^*\cup I_3^*\cup I_5^*$, and $\lambda_{1}^{G_{i}, k}\to \lambda_{1}^{G_{i}}\neq0$ for $i\in I_2^*$. Define
\begin{eqnarray*}
M^k = \max \left\{ \begin{array}{l}
\lambda_i^{k+1},\ i\in I_g(z^*);\ |\mu_i^{k+1}|,\ i\in \{1,2,\dots,q\};\
\xi_i^k,\ i\in I_1^*\cup I_3^*\cup I_4^*\cup I_5^*\cup I_6^*;\\
\phi_i^k,\ i\in I_1^*\cup I_4^*\cup I_6^*;\ \zeta_i^k,\ i\in I_2^*\cup I_3^*\cup I_4^*\cup I_5^*\cup I_6^*;\
\delta_i^k,\ i\in I_2^*\cup I_5^*\cup I_6^*;\ \eta^{k+1} 
\end{array} \right\}.
\end{eqnarray*}

We can show that the sequence of $\{M^k\}$ is bounded. Suppose by contradiction that $\{M^k\}$ is
unbounded. Taking a subsequence if necessary, we may assume $M^k \to +\infty$ as $k \to \infty$, then the following normalized multiplier sequences remain bounded:
\begin{align*}
&\left\{\tfrac{\lambda_i^{k+1}}{M^k}\right\},i \in I_g(z^*);
\ \left\{\tfrac{\mu_i^{k+1}}{M^k}\right\}, i \in \{1,2,\dots,q\};
\ \left\{\tfrac{\xi_i^k}{M^k}\right\}, i \in I_1^*\cup I_3^*\cup I_4^*\cup I_5^*\cup I_6^*;\\
& \left\{\tfrac{\phi_i^k}{M^k}\right\}, i \in I_1^*\cup I_4^*\cup I_6^*;\
\ \left\{\tfrac{\zeta_i^k}{M^k}\right\}, i\in I_2^*\cup I_3^*\cup I_4^*\cup I_5^*\cup I_6^*;
\ \left\{\tfrac{\delta_i^k}{M^k}\right\}, i\in I_2^*\cup I_5^*\cup I_6^*; 
\ \left\{\tfrac{\eta^{k+1} }{M^k}\right\}.
\end{align*}
By considering a convergent subsequence if necessary (without loss of generality), we may assume the following limits exist:
$$\begin{array}{l}
\tfrac{\lambda_i^{k+1}}{M^k}\to\lambda^*_i,i\in I_g(z^*);\ \tfrac{\mu_i^{k+1}}{M^k}\to\mu^*_i,i\in \{1,2,\dots,q\};\
\tfrac{\xi_i^k}{M^k}\to\xi^*_i,i\in I_1^*\cup I_3^*\cup I_4^*\cup I_5^*\cup I_6^*;\\ \tfrac{\phi_i^k}{M^k}\to\phi^*_i,i\in I_1^*\cup I_4^*\cup I_6^*;\ \tfrac{\zeta_i^k}{M^k}\to\zeta_i^*,i\in I_2^*\cup I_3^*\cup I_4^*\cup I_5^*\cup I_6^*;\ \tfrac{\delta_i^k}{M^k}\to\delta_i^*,i\in I_2^*\cup I_5^*\cup I_6^*;\\ \tfrac{\Lambda_i^k-\tilde{\Lambda}_i^k}{M^k}\to0;\ \tfrac{\Gamma_i^k-\tilde{\Gamma}_i^k}{M^k}\to0;\
\tfrac{\eta^{k+1}}{M^k}\to\eta^*;\ \tfrac{\eta^{k+1}\lambda_{1}^{G_{i}, k}}{M^k}\to \eta^*\lambda_{1}^{G_{i}}, i\in I_2^*;\\
\tfrac{\eta^{k+1}\lambda_{2}^{G_{i}, k}}{M^k}\to \eta^*\lambda_{2}^{G_{i}}, i\in I_2^*\cup I_3^*\cup I_5^*;\ \tfrac{\eta^{k+1}\lambda_{1}^{H_{i}, k}}{M^k}\to \eta^*\lambda_{1}^{H_{i}}, i\in I_1^*;\ \tfrac{\eta^{k+1}\lambda_{2}^{H_{i}, k}}{M^k}\to \eta^*\lambda_{2}^{H_{i}}, i\in I_1^*\cup I_3^*\cup I_4^*.
\end{array}$$

By definition of $\{M^k\}$, the limit vector $(\lambda^*, \mu^*, \xi^*, \phi^*, \zeta^*, \delta^*,\eta^*\lambda_{1}^{G_{i}}, \eta^*\lambda_{2}^{G_{i}},\eta^*\lambda_{1}^{H_{i}},\eta^*\lambda_{2}^{H_{i}}) \neq {0}$.
Since $\tfrac{\eta^{k+1}}{M^k} \to \eta^*$, the sequence $\{\tfrac{\eta^{k+1}}{M^k}\}$ is bounded. From the definition of $A^k$ and $B^k$, we have $\tfrac{A^k}{M^k} \to 0$ and $\tfrac{B^k}{M^k} \to 0$. Dividing equation \eqref{Equation-4} by $M^k$ and taking limits yields
\begin{align}\label{Equation-5}
0~=~&\nabla f(z^*)+\sum_{i\in I_g(z^*)}\lambda_i^*\nabla g_i(z^*) + \sum_{i=1}^q\mu_i^*\nabla h_i(z^*)
-\sum_{i\in I_1^*\cup I_4^*\cup I_6^*}\nabla G_i(z^*)(\xi_{i}^{*}(1,-\bar{w}_i)+\phi_{i}^{*}(1,\bar{w}_i))\nonumber \\
&-\sum_{i\in I_3^*\cup I_5^*}\xi_{i}^{*}\nabla G_i(z^*)u_{1}^{G_{i}}-\sum_{i\in I_2^*\cup I_5^*\cup I_6^*}\nabla H_i(z^*)(\zeta_{i}^{*}(1,-\bar{v}_i)+\delta_{i}^{*}(1,\bar{v}_i))\nonumber \\
&-\sum_{i\in I_3^*\cup I_4^*}\zeta_{i}^{*}\nabla H_i(z^*)u_{1}^{H_{i}}+\sum_{i\in I_1^*}\eta^*\lambda_{1}^{H_{i}}\nabla G_i(z^*)u_{1}^{H_{i}}+\sum_{i\in I_1^*\cup I_3^*\cup I_4^*}\eta^*\lambda_{2}^{H_{i}}\nabla G_i(z^*)u_{2}^{H_{i}}\nonumber \\
&+\sum_{i\in I_2^*}\eta^*\lambda_{1}^{G_{i}}\nabla H_i(z^*)u_{1}^{G_{i}}+\sum_{i\in I_2^*\cup I_3^*\cup I_5^*}\eta^*\lambda_{2}^{G_{i}}\nabla H_i(z^*)u_{2}^{G_{i}}.
\end{align}

Next, we will classify and discuss the values of $\eta^*$. If $\eta^*=0$, then $(\lambda^*, \mu^*, \xi^*, \phi^*, \zeta^*, \delta^*,0,0,0,0) \neq {0}$. Then \eqref{Equation-5} can be rewritten as
\begin{align*}
0~=~&\nabla f(z^*)+\sum_{i\in I_g(z^*)}\lambda_i^*\nabla g_i(z^*) + \sum_{i=1}^q\mu_i^*\nabla h_i(z^*)
-\sum_{i\in I_1^*\cup I_4^*\cup I_6^*}\nabla G_i(z^*)(\xi_{i}^{*}(1,-\bar{w}_i)+\phi_{i}^{*}(1,\bar{w}_i))\nonumber \\
&-\sum_{i\in I_3^*\cup I_5^*}\xi_{i}^{*}\nabla G_i(z^*)u_{1}^{G_{i}}-\sum_{i\in I_2^*\cup I_5^*\cup I_6^*}\nabla H_i(z^*)(\zeta_{i}^{*}(1,-\bar{v}_i)+\delta_{i}^{*}(1,\bar{v}_i))-\sum_{i\in I_3^*\cup I_4^*}\zeta_{i}^{*}\nabla H_i(z^*)u_{1}^{H_{i}}.
\end{align*}
This means that \eqref{SOCMPCC-weak-nondegeneracy} is positively linearly dependent, which contradicts the SOCMPCC-weak-Robinson's CQ.

If $\eta^* \neq0$, then $\eta^*\lambda_{1}^{H_{i}}>0$, $\eta^*\lambda_{2}^{H_{i}}>0$, $\eta^*\lambda_{1}^{G_{i}}>0$, and $\eta^*\lambda_{2}^{G_{i}}>0$.
Applying Lemma \ref{equivalent} and passing to a subsequence, we obtain that $G_i(z^*)$ and $H_i(z^*)$ operator commute for every $i\in I_1^*$. Consequently, 
\begin{align*}
&-\sum_{i\in I_1^*}\nabla G_i(z^*)(\xi_{i}^{*}(1,-\bar{w}_i)+\phi_{i}^{*}(1,\bar{w}_i))+\sum_{i\in I_1^*}\eta^*\lambda_{1}^{H_{i}}\nabla G_i(z^*)u_{1}^{H_{i}}+\sum_{i\in I_1^*}\eta^*\lambda_{2}^{H_{i}}\nabla G_i(z^*)u_{2}^{H_{i}}\\
=&-\sum_{i\in I_1^*}\nabla G_i(z^*)(\bar\xi_{i}^{*}(1,-\bar{w}_i)+\bar\phi_{i}^{*}(1,\bar{w}_i)),
\end{align*}
where $\bar\xi_{i}^{*}=\xi_{i}^{*}-\tfrac{1}{2}\eta^*\lambda_{1}^{H_{i}}(i\in I_1^*)$, $\bar\phi_{i}^{*}=\phi_{i}^{*}-\tfrac{1}{2}\eta^*\lambda_{2}^{H_{i}}(i\in I_1^*)$. Applying Lemma \ref{cone complementarity-lemma}, we obtain
\begin{align*}
-\sum_{i\in I_3^*}\xi_{i}^{*}\nabla G_i(z^*)u_{1}^{G_{i}}+\sum_{i\in I_3^*}\eta^*\lambda_{2}^{H_{i}}\nabla G_i(z^*)u_{2}^{H_{i}}=-\sum_{i\in I_3^*}\tilde\xi_{i}^{*}\nabla G_i(z^*)u_{1}^{G_{i}},
\end{align*}
where $\tilde\xi_{i}^{*}=\xi_{i}^{*}-\eta^*\lambda_{2}^{H_{i}}(i\in I_3^*)$. Following the same reasoning as above, \eqref{Equation-5} can be rewritten as
\begin{align}\label{Equation-6}
0~=~&\nabla f(z^*)+\sum_{i\in I_g(z^*)}\lambda_i^*\nabla g_i(z^*) + \sum_{i=1}^q\mu_i^*\nabla h_i(z^*)-\sum_{i\in I_1^*}\nabla G_i(z^*)(\bar\xi_{i}^{*}(1,-\bar{w}_i)+\bar\phi_{i}^{*}(1,\bar{w}_i))\nonumber \\
&-\sum_{i\in I_4^*}\nabla G_i(z^*)(\tilde\xi_{i}^{*}(1,-\bar{w}_i)+\tilde\phi_{i}^{*}(1,\bar{w}_i))
-\sum_{i\in I_6^*}\nabla G_i(z^*)(\xi_{i}^{*}(1,-\bar{w}_i)+\phi_{i}^{*}(1,\bar{w}_i))\nonumber\\
&-\sum_{i\in I_3^*}\tilde\xi_{i}^{*}\nabla G_i(z^*)u_{1}^{G_{i}}-\sum_{i\in I_5^*}\xi_{i}^{*}\nabla G_i(z^*)u_{1}^{G_{i}}-\sum_{i\in I_2^*}\nabla H_i(z^*)(\bar\zeta_{i}^{*}(1,-\bar{v}_i)+\bar\delta_{i}^{*}(1,\bar{v}_i))\nonumber\\
&-\sum_{i\in I_5^*}\nabla H_i(z^*)(\tilde\zeta_{i}^{*}(1,-\bar{v}_i)+\tilde\delta_{i}^{*} (1,\bar{v}_i))-\sum_{i\in I_6^*}\nabla H_i(z^*)(\zeta_{i}^{*}(1,-\bar{v}_i)+\delta_{i}^{*}(1,\bar{v}_i))\nonumber \\
&-\sum_{i\in I_3^*}\tilde\zeta_{i}^{*}\nabla H_i(z^*)u_{1}^{H_{i}}-\sum_{i\in I_4^*}\zeta_{i}^{*}\nabla H_i(z^*)u_{1}^{H_{i}},
\end{align}
where $\tilde\xi_{i}^{*}=\xi_{i}^{*}(i\in I_4^*)$ and $\tilde\phi_{i}^{*}=\phi_{i}^{*}-\tfrac{1}{2}\eta^*\lambda_{2}^{H_{i}}(i\in I_4^*)$, 
$\bar\zeta_{i}^{*}=\zeta_{i}^{*}-\tfrac{1}{2}\eta^*\lambda_{1}^{G_{i}}(i\in I_2^*)$,
$\bar\delta_{i}^{*}=\delta_{i}^{*}-\tfrac{1}{2}\eta^*\lambda_{2}^{G_{i}}(i\in I_2^*)$,
$\tilde\zeta_{i}^{*}=\zeta_{i}^{*}(i\in I_5^*)$, $\tilde\delta_{i}^{*}=\delta_{i}^{*}-\tfrac{1}{2}\eta^*\lambda_{2}^{G_{i}}(i\in I_5^*)$, $\tilde\zeta_{i}^{*}=\zeta_{i}^{*}-\eta^*\lambda_{2}^{G_{i}}(i\in I_3^*)$.

Next, we prove that \eqref{Equation-6} has a nontrivial solution, which contradicts the SOCMPCC-weak-Robinson's CQ. To this end, we further distinguish whether $\tilde\xi_{i}^{*}=0(i\in I_4^*)$ or not. For any $i\in I_4^*$, if $\tilde\xi_{i}^{*}\neq 0$, then system \eqref{Equation-6} admits a nonzero solution, which contradicts the SOCMPCC-weak-Robinson's CQ. It remains to consider the case where $\tilde\xi_{i}^{*}=0$. In this case, it follows from \eqref{decomposition} that $\lambda_1^{G_i,k}\leq\lambda_2^{G_i,k}$, and moreover, as $z^k\to z^*$, $\lambda_1^{G_i,k},\lambda_2^{G_i,k}\to 0$. Next, we distinguish the following cases.
\begin{itemize}
  \item  $\lambda_1^{G_i,k}\geq0$. Hence, $\xi_{i}^{k}=(-\rho_k\lambda_1^{G_{i},k})_{+}=0$, $\phi_{i}^{k}=(-\rho_k\lambda_2^{G_{i},k})_{+}=0$, $\tilde\xi_{i}^{*}=0$, and $\tilde\phi_{i}^{*}=-\tfrac{1}{2}\eta^*\lambda_{2}^{H_{i}}<0$. Therefore, \eqref{Equation-6} has a nontrivial solution, contradicting the SOCMPCC-weak-Robinson's CQ. 
  \item  $\lambda_1^{G_i,k}<0$.
\begin{itemize}
    \item  $\lambda_2^{G_i,k}\geq0$. Hence, $\phi_{i}^{k}=(-\rho_k\lambda_2^{G_{i},k})_{+}=0$ and $\tilde\phi_{i}^{*}=-\tfrac{1}{2}\eta^*\lambda_{2}^{H_{i}}<0$. Therefore, \eqref{Equation-6} has a nontrivial solution, contradicting the SOCMPCC-weak-Robinson's CQ.
    \item  $\lambda_2^{G_i,k}<0$. Since $\lambda_1^{G_i,k}\leq\lambda_2^{G_i,k}<0$, $0<(-\rho_k\lambda_2^{G_{i},k})_{+}\leq\xi_i^k=(-\rho_k\lambda_1^{G_{i},k})_{+}$. Dividing by $M^k$ and taking the limit as $k\to \infty$, since $\tilde\xi_{i}^{*}=0$, we obtain $\tfrac{(-\rho_k\lambda_2^{G_{i},k})_{+}}{M^k}\to 0$. Hence, $\tilde\phi_{i}^{*}=-\tfrac{1}{2}\eta^*\lambda_{2}^{H_{i}}<0$. Therefore, \eqref{Equation-6} has a nontrivial solution, contradicting the SOCMPCC-weak-Robinson's CQ. 
        \end{itemize}
\end{itemize}

By the above proof by contradiction, assuming that $\{{M^k}\}$ is unbounded leads to a contradiction with the SOCMPCC-weak-Robinson's CQ. Hence, $\{{M^k}\}$ is bounded. This means that the following normalized multiplier sequences remain bounded:
\begin{align*}
&\left\{\lambda_i^{k+1}\right\}, i \in I_g(z^*);
\ \left\{\mu_i^{k+1}\right\}, i \in \{1,2,\dots,q\}; 
\ \left\{\xi_i^k\right\}, i \in I_1^*\cup I_3^*\cup I_4^*\cup I_5^*\cup I_6^*;\\
& \left\{\phi_i^k\right\}, i \in I_1^*\cup I_4^*\cup I_6^*;
\ \left\{\zeta_i^k\right\}, i\in I_2^*\cup I_3^*\cup I_4^*\cup I_5^*\cup I_6^*;
\ \left\{\delta_i^k\right\}, i\in I_2^*\cup I_5^*\cup I_6^*; 
\ \left\{\eta^{k+1}\right\}.
\end{align*}

For $\tilde\Lambda_i^k$, we have $\tilde{\Lambda}_i^k=-(\xi_i^k u_1^{G_{i},k} + \phi_i^k u_2^{G_{i},k})$. Since $\{\xi_i^k\}$ and $\{\phi_i^k\}$ are bounded, while $u_1^{G_{i},k}$ and $u_2^{G_{i},k}$ converge, it follows that $\{\tilde\Lambda_i^k\}$ is bounded. Moreover, since $\{\Lambda_i^k-\tilde{\Lambda}_i^k\}$ is bounded, we obtain $\Lambda_i^k=(\Lambda_i^k-\tilde{\Lambda}_i^k)+\tilde\Lambda_i^k$, and hence $\{\Lambda_i^k\}$ is bounded as well. Similarly, $\{\Gamma_i^k\}$ is bounded. We conclude that the original multipliers $(\lambda^{k+1},\mu^{k+1},\eta^{k+1},\Lambda^k,\Gamma^k)$ are also bounded. By taking convergent subsequences if necessary, we may assume that
$$\begin{array}{l}
\lambda_i^{k+1}\to\lambda^*_i,i\in I_g(z^*);\ \mu_i^{k+1}\to\mu^*_i,i\in \{1,2,\dots,q\};\\ \Lambda_i^k\to\Lambda^*_i,i\in \{1,2,\dots,J\};\ \Gamma_i^k\to\Gamma^*_i,i\in \{1,2,\dots,J\};\
\eta^{k+1}\to\eta^*.
\end{array}$$

Taking limits in \eqref{Equation-2} yields
\begin{align*}
0~=~& \nabla f(z^*) + \sum\limits_{i\in I_g(z^*)}\lambda_i^*\nabla g_i(z^*) + \sum\limits_{i=1}^q\mu_i^*\nabla h_i(z^*) + \sum\limits_{i=1}^J\nabla G_i(z^*)\Lambda_i^*\\
&+ \sum\limits_{i=1}^J\nabla H_i(z^*)\Gamma_i^*+\eta^*\nabla (G(z^*)^TH(z^*)).
\end{align*}

The feasible point $z^*$ implies that $g_i(z^*)\leq 0$, $G_i(z^*)\in \mathcal{K}_i$, and $H_i(z^*)\in \mathcal{K}_i$ hold for all constraints. From the definition of $\lambda_i^{k+1}$ in Algorithm \ref{alg:socmpcc}, when $\{M^k\}$ is bounded, $\lambda_i^*=0$ holds whenever $g_i(z^*)<0$. Consequently, the complementarity condition $g(z^*)^T\lambda^*=0$ is satisfied when $\{M^k\}$ is bounded.

By the definition of ${\Lambda}^k$ and Lemma \ref{lema3}, $\langle{\Lambda}^k,\Pi_{{\cal K}}(G(z^k)+\tfrac{\bar{\Lambda}^k}{\rho_k}\rangle=0$ holds for all $k$. Using the Cauchy–Schwarz inequality and the nonexpansiveness of the projection, we obtain
\begin{align*}
|\langle\Lambda^k,\Pi_{\mathcal{K}}(G(z^k))\rangle|
~=~&|\langle\Lambda^k,\Pi_{\mathcal{K}}(G(z^k))-\Pi_{{\cal K}}(G(z^k)+\tfrac{\bar{\Lambda}^k}{\rho_k}\rangle|\\
~\leq~& \|\Lambda^k\| \cdot \|\Pi_{\mathcal{K}}(G(z^k))-\Pi_{{\cal K}}(G(z^k)+\tfrac{\bar{\Lambda}^k}{\rho_k}\|\\
~\leq~& \|\Lambda^k\|\cdot \tfrac{\|\bar\Lambda^k\|}{\rho_k} \\
~=~&\|\Pi_{\mathcal{K}^\circ}(G(z^k)+\tfrac{\bar{\Lambda}^k}{\rho_k})\|\cdot\|\bar\Lambda^k\|
\end{align*}
Since $\Pi_{\mathcal{K}^\circ}(G(z^k)+\tfrac{\bar{\Lambda}^k}{\rho_k})\to 0$, it follows that $\langle{\Lambda}^k,\Pi_{\cal K}(G(z^k))\rangle\to 0$. Consequently, $\langle{\Lambda}^*,G(z^*)\rangle=0$.
Similarly, we have $\langle{\Gamma}^*,H(z^*)\rangle=0$. From the definitions of ${\Lambda}^k$, ${\Gamma}^k$, and ${\eta^{k+1}}$ in Algorithm \ref{alg:socmpcc}, we conclude ${\Lambda}^*,{\Gamma}^*\in{\cal K}^\circ$ and $\eta^*\geq0$.

In summary, there exist multipliers $(\lambda^*,\mu^*,\Lambda^*,\Gamma^*,\eta^*)$ satisfying the following KKT condition:
\begin{equation*}
\begin{cases}
\nabla f(z^*)+\nabla g(z^*)\lambda^*+\nabla h(z^*)\mu^*+\sum\limits_{i=1}^J\nabla G_i(z^*)\Lambda_i^*+\sum\limits_{i=1}^J\nabla H_i(z^*)\Gamma_i^*+\nabla(G^TH)(z^*)\eta^*=0 \\
\lambda^*\geq 0,\ g(z^*)^T\lambda^*=0, \\
G_i(z^*)\in {\cal K}_i, -\Lambda^*_i\in {\cal K}_i, G_i(z^*)^T\Lambda^*_i=0, \quad i=1,\cdots,J, \\
H_i(z^*)\in {\cal K}_i, -\Gamma^*_i\in {\cal K}_i, H_i(z^*)^T\Gamma^*_i=0, \quad i=1,\cdots,J, \\
\eta^*\geq 0.
\end{cases}
\end{equation*}
Therefore, when $\{\rho_k\}$ is unbounded, $z^*$ is a K-stationary point of SOCMPCC.
\end{proof}

\begin{remark}\rm
Theorems \ref{convergence-thm1} and \ref{convergence-thm2} show that when $\{\rho_k\}$ is bounded, Theorem \ref{feasibility-thm} guarantees that any limit point $z^*$ of the sequence $\{z^k\}$ generated by Algorithm \ref{alg:socmpcc} is feasible, and Theorem \ref{convergence-thm1} further ensures that $z^*$ is a K-stationary point of the SOCMPCC; when $\{\rho_k\}$ is unbounded, if $z^*$ is feasible (e.g., the conditions of Theorem \ref{feasibility-thm} hold) and the SOCMPCC-weak-Robinson's CQ holds, then $z^*$ is a K-stationary point of the SOCMPCC.
\end{remark}

}

\section{Numerical experiments}

In this section, we present numerical results obtained by implementing Algorithm \ref{alg:socmpcc}.
All computations were carried out in MATLAB R2021b on a personal computer equipped with an 11th Gen Intel(R) Core(TM) i5-1135G7 CPU at 2.40 GHz and 16 GB of memory.
Throughout the experiments, ${\cal K}^m$ denotes the $m$-dimensional second-order cone.

The following parameters are used in Algorithm \ref{alg:socmpcc}:
$
\tau=3,\ \sigma=0.2,\ \rho_0=2,\ \lambda_{\max}=10^{20},\ \mu_{\max}=10^{20},\ \mu_{\min}=-10^{20},\ \eta_{\max}=10^{20},\ \eta_{\min}=-10^{20},\ {\varepsilon=10^{-6}}.
$
{ In all experiments, we set a lower bound of $10^{-6}$ for $\epsilon_k$, meaning that $\epsilon_k\geq 10^{-6}$.}
The initial multipliers are chosen as $\bar\lambda^0=1$, $\bar\mu^0=1$, $\bar\eta^0=1$, $\bar\Lambda^0=(-1,1,0)$, and $\bar\Gamma^0=(-1,-1,0)$.
For Examples \ref{ex 1} and \ref{ex 5}, we set $\bar\Lambda^0=(-1,1)$ and $\bar\Gamma^0=(-1,-1)$.
In Example \ref{ex 5}, since the inequality constraint is two-dimensional, we take $\bar\lambda^0=(1,1)$.

%%%%%%%%%%%%%%%%
\begin{ex}\label{ex 1}\rm
Consider the SOCMPCC:
\begin{eqnarray*}
\min   & & f(z)=-z_1+z^2_2 \\
\mbox{\rm s.t.} && G(z)=\begin{pmatrix}z_2+1\\z_1+1\\ \end{pmatrix}\in {\cal K}^2,~
H(z)=\begin{pmatrix}z_2\\z_1\\ \end{pmatrix}\in {\cal K}^2,\\
&&G(z)^T H(z)=0.\nonumber
\end{eqnarray*}
The optimal solution is $z^* = (0, 0)^\top$, with optimal value $f(z^*) = 0$.
\end{ex}

%%%%%%%%%%%%%%%
\begin{ex}{\rm \cite{Y.-C. Liang Y.-W. Liu}}\rm\label{ex 3}
Consider the SOCMPCC:
\begin{eqnarray*}
\min   & & f(z)=z^2_1+z^2_2+z^2_3 \\
\mbox{s.t.}
&& G(z)=\begin{pmatrix}-z_1+1\\z_2+1\\z_3\\ \end{pmatrix}\in {\cal K}^3,~
 H(z)=\begin{pmatrix}z_1+1\\z_2-1\\z_3\\ \end{pmatrix}\in {\cal K}^3,\\
&&G(z)^T H(z)=0.\nonumber
\end{eqnarray*}
The optimal solution is $z^* = (0, 0, 0)^\top$, with optimal value $f(z^*) = 0$.
\end{ex}

%%%%%%%%%%%%%%%
\begin{ex}\rm\cite{Zhang Wu Zhang}\label{ex 4}
Consider the SOCMPCC:
\begin{eqnarray*}
\min   & & f(z)=2z_1 - 3z_3 \\
\mbox{\rm s.t.}
&& G(z)=\begin{pmatrix}1 + z_1 \\z_1 - z_2 - 2z_3 + 1 \\z_3 - 3\end{pmatrix}\in \mathcal{K}^3,~
 H(z) = \begin{pmatrix}z_1 \\z_2 \\z_3\end{pmatrix} \in \mathcal{K}^3,\\
&&G(z)^T H(z)=0.\nonumber
\end{eqnarray*}
The optimal solution is $z^* = (1, 0, 1)^\top$, with optimal value $f(z^*) = -1$.
\end{ex}

%%%%%%%%%%%%%%%
\begin{ex}\rm\cite{Zhang Wu Zhang}\label{ex 5}
Consider the SOCMPCC:
\begin{eqnarray*}
\min   & & f(z) = z_1 + z_2 - z_3 \\
\mbox{\rm s.t.}
&& g(z) = \begin{pmatrix} z_2 \\ z_3 \end{pmatrix} \leq 0, \\
&& G(z) =\begin{pmatrix} z_1 - z_3 + 1 \\ -z_2 - 1 \end{pmatrix}\in \mathcal{K}^2,~
 H(z) = \begin{pmatrix} z_2 \\ 2z_1 \end{pmatrix} \in \mathcal{K}^2,\\
&& G(z)^T H(z)=0.\nonumber
\end{eqnarray*}
The optimal solution is $z^* = (0, 0, 0)^\top$, with optimal value $f(z^*) = 0$.
\end{ex}

To evaluate the performance of Algorithm \ref{alg:socmpcc}, we compare the proposed augmented Lagrangian method with several smoothing methods. Specifically, SM1 denotes the method in \cite{zhu-pang-lin}, while SM2--SM5 correspond to smoothing functions proposed in \cite{Liang-Xia}. The initial smoothing parameter is set to $t_0=10^{-2}$ and updated via $t_{k+1}=0.1t_k$.

The computational results are summarized in Tables \ref{tab1}--\ref{tab4}.

%%%%%%%%%%%%%%%%%%%%%%%%%%%%%%%%%

\FloatBarrier
\begin{table}[!htp]
    \centering
    \caption{Computational results of Example \ref{ex 1}}
    \begin{adjustbox}{width=\linewidth}
    \begin{tabular}{cccccccc}
       \toprule
        & $z^k$ & $f(z^k)$ & $\|z^k-z^*\|$ & $|f(z^k)-f(z^*)|$ & times(s)& $z^*$ & $f(z^*)$   \\
         \midrule
        ALM &1.0e-06$\ *\begin{pmatrix}-0.0988 \\0.1279\end{pmatrix}$ & 9.8831e-08 & 1.6166e-07 & 9.8831e-08 & 0.0370&{\multirow{10}{*}{$\begin{pmatrix} 0\\0 \end{pmatrix}$}}&{\multirow{10}{*}0}   \\

        SM1 & 1.0e-06$\ *\begin{pmatrix} -0.5000 \\0.5000\end{pmatrix}$&
        5.0002e-07 & 7.0713e-07 & 5.0002e-07 & 0.0406  \\

        SM2 & 1.0e-06$\ *\begin{pmatrix} -0.1499 \\ -0.0501\end{pmatrix}$&
        1.4987e-07 & 1.5803e-07 & 1.4987e-07 & 0.0469   \\

        SM3 & 1.0e-06$\ *\begin{pmatrix}  -0.1501 \\ -0.0499\end{pmatrix}$&
        1.5006e-07 & 1.5815e-07 & 1.5006e-07 & 0.0540  \\

        SM4 & 1.0e-06$\ *\begin{pmatrix}  -0.1497 \\ -0.0503\end{pmatrix}$&
        1.4975e-07 & 1.5795e-07 & 1.4975e-07 & 0.0403\\

        SM5 & 1.0e-06$\ *\begin{pmatrix}  -0.1500 \\ -0.0500\end{pmatrix}$&
        1.4998e-07 & 1.5810e-07 & 1.4998e-07 & 0.0338 \\
 \bottomrule
 \end{tabular}\label{tab1}
 \end{adjustbox}
 \end{table}

\FloatBarrier
\begin{table}[!htp]
    \centering
    \caption{Computational results of Example \ref{ex 3}}
    \begin{adjustbox}{width=\linewidth}
    \begin{tabular}{cccccccc}
       \toprule
        & $z^k$ & $f(z^k)$ & $\|z^k-z^*\|$ & $|f(z^k)-f(z^*)|$ & times(s) & $z^*$ & $f(z^*)$   \\
         \midrule
        ALM & 1.0e-08$\ *\begin{pmatrix}  -0.1036\\0.0885 \\-0.1416\end{pmatrix}$&
         3.8613e-18 & 1.9650e-09  & 3.8613e-18 & 0.0060&{\multirow{14}{*}{$\begin{pmatrix} 0\\0\\0 \end{pmatrix}$}}&{\multirow{14}{*}0}   \\

        SM1 & 1.0e-08$\ *\begin{pmatrix}  0.5182\\ -0.5181  \\ -0.0000\end{pmatrix}$& 5.3693e-17 & 7.3276e-09 & 5.3693e-17 & 0.0284   \\

        SM2 & $\begin{pmatrix}1.0000 \\-1.0000 \\ -0.0000
          \end{pmatrix}$& 2.0000 & 1.4142 & 2.0000 & 0.0299   \\

        SM3 & $\begin{pmatrix}1.0000 \\-1.0000 \\ 0.0000
          \end{pmatrix}$& 2.0000 & 1.4142 & 2.0000 & 0.0298  \\

        SM4 & $\begin{pmatrix}1.0000 \\-1.0000 \\   -0.0000
          \end{pmatrix}$& 2.0000 & 1.4142 & 2.0000 &  0.0407\\

        SM5 & $\begin{pmatrix}1.0000 \\-1.0000\\0.0000
          \end{pmatrix}$& 2.0000 & 1.4142 & 2.0000 & 0.0309 \\
 \bottomrule
 \end{tabular}\label{tab2}
 \end{adjustbox}
 \end{table}

\FloatBarrier
\begin{table}[!htp]
    \centering
    \caption{Computational results of Example \ref{ex 4}}
    \begin{adjustbox}{width=\linewidth}
    \begin{tabular}{c>{\centering\arraybackslash}p{3.5cm}>{\centering\arraybackslash}p{3cm}ccccc}
       \toprule
        & $z^k$ & $f(z^k)$ & $\|z^k-z^*\|$ & $|f(z^k)-f(z^*)|$ & times(s) & $z^*$ & $f(z^*)$   \\
         \midrule
        ALM & $\begin{pmatrix}1.0000\\0.0000\\1.0000\end{pmatrix}$&
        -1.0000 &  3.6504e-08   & 3.3493e-08  &  0.5996&{\multirow{14}{*}{$\begin{pmatrix} 1\\0\\1 \end{pmatrix}$}}&{\multirow{14}{*}{-1}}   \\

        SM1 & $\begin{pmatrix}1.0000\\-0.0000\\1.0000\end{pmatrix}$&
        -1.0000 &  4.1458e-07   & 3.7500e-07   &  0.6982   \\

        SM2 & $\begin{pmatrix}1.0000 \\ 0.0000 \\1.0000\end{pmatrix}$&
        -1.0000 & 5.9161e-07 & 9.0000e-07   &  0.1296   \\

        SM3 & $\begin{pmatrix}1.0000 \\0.0000 \\1.0000\end{pmatrix}$&
        -1.0000 & 5.9161e-07 & 9.0000e-07 & 0.0385  \\

        SM4 & $\begin{pmatrix}1.0000 \\0.0000 \\1.0000\end{pmatrix}$&
        -1.0000 & 5.9161e-07 & 9.0000e-07 &   0.0415\\

        SM5 & $\begin{pmatrix}1.0000 \\0.0000 \\1.0000\end{pmatrix}$&
        -1.0000 & 5.9161e-07 & 9.0000e-07 & 0.0450 \\
 \bottomrule
 \end{tabular}\label{tab3}
 \end{adjustbox}
 \end{table}

\FloatBarrier
\begin{table}[!htp]
    \caption{Computational results of Example \ref{ex 5}}
    \centering
    \begin{adjustbox}{width=\linewidth}
    \begin{tabular}{cccccccc}
     \toprule
       & $z^k$ & $f(z^k)$ & $\|z^k-z^*\|$ & $|f(z^k)-f(z^*)|$ & times(s) & $z^*$ & $f(z^*)$   \\
         \midrule
      ALM & 1.0e-08$\ *\begin{pmatrix} 0.0644\\-0.1932 \\0.2211\end{pmatrix}$&
        -3.4991e-09&  3.0061e-09 &  3.4991e-09& 0.0472&{\multirow{16}{*}{$\begin{pmatrix} 0\\0\\0 \end{pmatrix}$}}&{\multirow{16}{*}0}   \\

        SM1 & 1.0e-04$\ *\begin{pmatrix}  -0.0000 \\  -0.0000 \\-0.2063\end{pmatrix}$&  2.0628e-05 &  2.0628e-05 & 2.0628e-05 & 0.4661  \\

        SM2 & 1.0e-06$\ *\begin{pmatrix} -0.0007\\-0.0058 \\-0.1534\end{pmatrix}$&
        1.4698e-07 & 1.5351e-07 & 1.4698e-07 & 0.1570  \\

        SM3 &1.0e-05$\ *\begin{pmatrix} -0.0000\\ -0.0000 \\-0.8903\end{pmatrix}$&
        8.9032e-06 & 8.9033e-06 & 8.9032e-06 & 5.1042\\

        SM4 &1.0e-07$\ *\begin{pmatrix}-0.0030\\-0.0289 \\-0.2998\end{pmatrix}$&
        2.6792e-08 & 3.0115e-08 &  2.6792e-08 & 0.1735 \\

        SM5 & 1.0e-05$\ *\begin{pmatrix}0.0051\\-0.0098\\-0.2046\end{pmatrix}$&
        1.9992e-06 & 2.0488e-06 & 1.9992e-06 & 0.1268\\
 \bottomrule
 \end{tabular}\label{tab4}
 \end{adjustbox}
 \end{table}
 \FloatBarrier

Tables \ref{tab1}--\ref{tab4} report the numerical results with four-digit precision. In these tables, $z^*$ and $f(z^*)$ denote the true optimal solution and optimal value, respectively, while $z^k$ and $f(z^k)$ represent the computed solution and objective value. The quantities $\|z^k - z^*\|$ and $|f(z^k) - f(z^*)|$ measure the absolute errors in the solution and objective value, respectively.

From the tables, it can be observed that both the smoothing methods and the proposed ALM are capable of obtaining solutions close to the true optimum. However, across all test instances, the ALM consistently yields smaller absolute errors in both the decision variables and objective values. This indicates that the proposed method achieves higher solution accuracy and produces results closer to the theoretical optimum.

For a more intuitive comparison, Fig. \ref{fig1} illustrates the absolute errors (on a logarithmic scale) for different methods. Lower values correspond to better performance, and the figure clearly shows the advantage of the proposed ALM in terms of accuracy.

\FloatBarrier
\begin{figure}[htbp]
 \centering
\includegraphics[width=\textwidth]{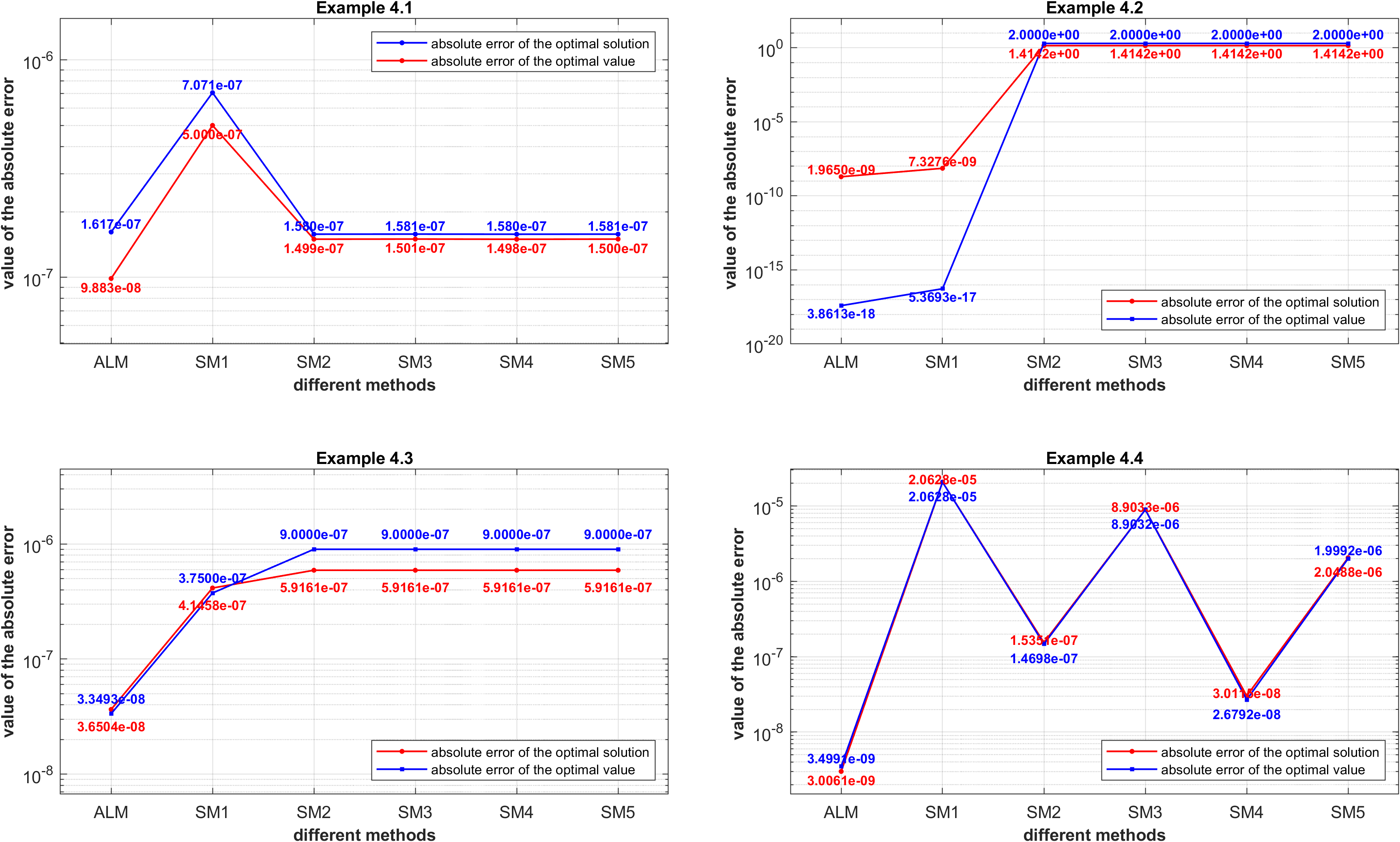}
 \caption{\scriptsize Comparison of Absolute Errors Between the Optimal Solution and the Optimal Value (Logarithmic Scale)}
 \label{fig1}
\end{figure}
\FloatBarrier

%%%%%%%%%%%%%%%%%%%%%%%%%%%

To further assess scalability, we apply the proposed method to high-dimensional SOCMPCCs.

\begin{ex}\rm\cite{Yamamura Okuno Hayashi Fukushima } \label{high-exam}
Consider the SOCMPCC:
\begin{eqnarray*}
\min   & & \|x\|^2 + \|y\|^2 \\
\mbox{\rm s.t.}
&& Ax \le b,\\
&& z = Nx + My + q,\\
&& \mathcal{K} \ni y \perp z \in \mathcal{K},
\end{eqnarray*}
{
where $(x,y,z) \in \mathbb{R}^l \times \mathbb{R}^m \times \mathbb{R}^m$.} The entries of $A$ and $N$ are randomly generated from $[-1,1]$, while $b$ is sampled from $[0,1]$. The matrix $M$ is symmetric positive definite, constructed as $M = M_1 M_1^\top + 0.01I$, where $M_1$ has entries drawn from $[-1,1]$. The vector $q$ is defined by $q := \xi_z - M\xi_y$, where $\xi_y$ and $\xi_z$ are randomly generated from $[-1,1]$.
\end{ex}

We initialize $w^0=(x^0,y^0,z^0)=(0,\xi_y,\xi_z)$, which satisfies the constraints. The parameters are chosen as { $\tau=2$, $\sigma=0.2$, $\rho_0=2$,} $\bar\lambda^0=(1,\dots,1)$, $\bar\mu^0=(1,\dots,1)$, $\bar\eta^0=1$, $\bar\Lambda^0=(-1,1,0,\dots,0)$, and $\bar\Gamma^0=(-1,-1,0,\dots,0)$.
{
The experimental data in \cite{Liang-Xia} show that when $t_k$ decreases too rapidly, the number of successful runs decreases. Therefore, we initialize the smoothing parameter as $t_0=10^{-2}$ and then update it according to $t_{k+1}=0.5t_k$.}

{
For each fixed $l$ or $m$, we randomly conducted 30 independent experiments. Since the two methods solve different subproblems, and the ALM may produce infeasible points while the smoothing methods may fail to converge, in our experiments, we mainly record data related to the core indicator: the feasibility of the ALM solutions and the convergence of the SM1-SM5 (see Fig. \ref{fig2}). For the problems where the ALM is feasible and the smoothing methods converge, we further compare the objective values obtained by the two methods. The results show that, on these problems, the ALM yields smaller objective values.
}

{
Fig. 2 presents the experimental results for $l=10$ and $l=100$ on randomly generated test sets with $m=20,30,40,50,60$. In each heatmap, the rows correspond to the cone constraint dimension $m$, and the columns correspond to different algorithms (ALM, SM1–SM5). Each cell reports the number of successful runs out of 30 independent trials and the corresponding percentage, with darker colors indicating higher success rates. Specifically, the ALM cells record the feasibility rate, while the SM1–SM5 cells record the convergence rate.}

\FloatBarrier
\begin{figure}[htbp]
 \centering
\includegraphics[width=\textwidth]{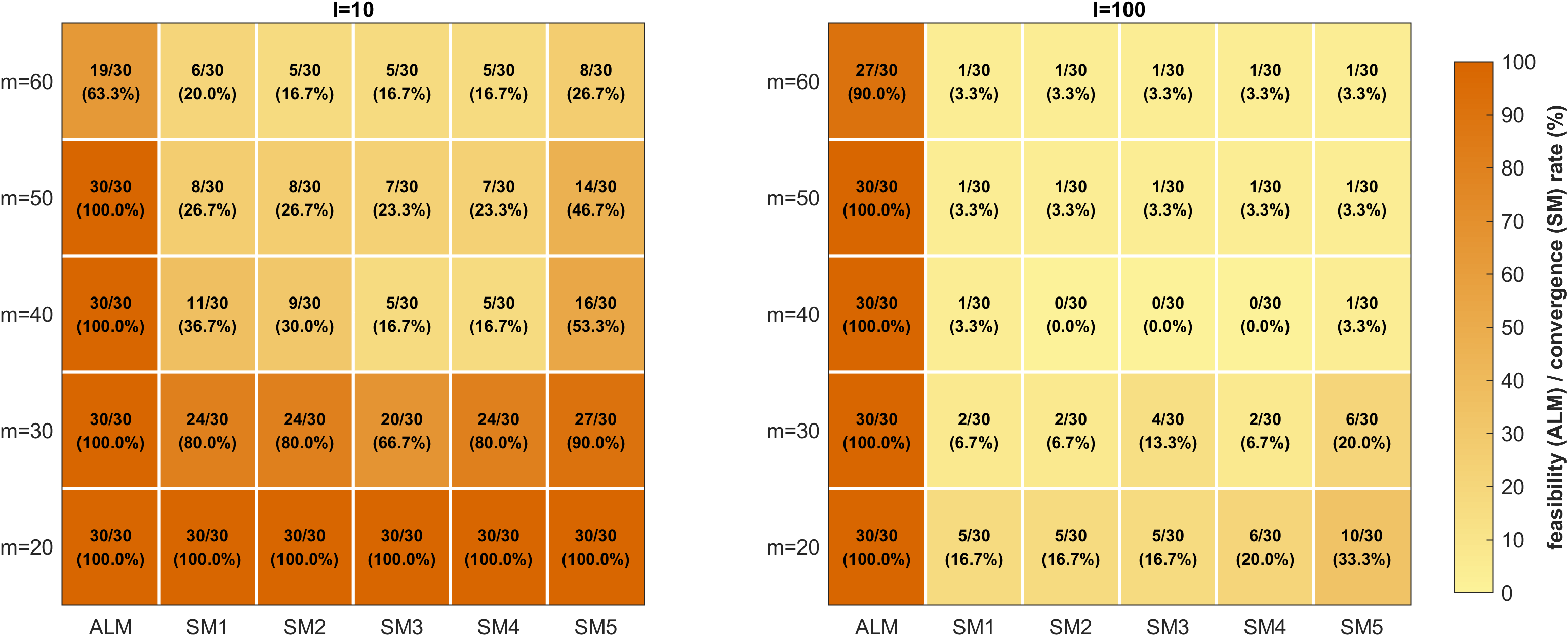}
 \caption{\scriptsize  Numbers of feasible runs for ALM and converged runs for smoothing methods, along with the corresponding feasibility rate and convergence rate, across different dimensions}
 \label{fig2}
\end{figure}
\FloatBarrier

{
It can be observed that the ALM maintains a relatively high feasibility rate across most tested dimensions, and its decline with increasing dimension is significantly smaller than that of the smoothing methods. For $l=10$ and $l=100$, the feasibility rate of the ALM is reaches 100\% for $m\leq50$, and reaches 63.3\% and 90\%, respectively, at $m=60$. Specifically, when $l$ is fixed, the feasibility rate of the ALM decreases slowly as $m$ increases. When $m$ is fixed, however, increasing $l$ does not cause a noticeable decay in the feasibility rate, indicating that the algorithm is robust to the growth of the ordinary decision variable dimension. In contrast, the convergence rates of the five smoothing methods SM1–SM5 are highly sensitive to both $l$ and $m$. Whether $m$ or $l$ increases, the convergence rates of the smoothing methods drop significantly; at $m=60$, the convergence rates of SM1–SM5 are 16.7\%–26.7\% for $l=10$, and almost all methods fail to converge for $l=100$. Notably, at $l=100$ and $m=60$, the feasibility rate of the ALM (90\%) is 27 times that of the best-performing SM (3.3\%).
}

{
These results show that, when $l$ is fixed, the convergence rates of SM1–SM5 decrease significantly as $m$ increases, while the feasibility rate of the ALM decreases only slightly but remains significantly higher than the convergence rates of the smoothing methods. When $m$ is fixed and $l$ increases from 10 to 100, the overall problem size grows, and the ALM still maintains stable solution performance, demonstrating good scalability.
}

In summary, the numerical experiments demonstrate that Algorithm \ref{alg:socmpcc} reliably converges to feasible solutions, achieves high approximation accuracy, and maintains competitive computational efficiency across a range of problem settings. These results confirm the effectiveness and robustness of the proposed augmented Lagrangian framework for solving SOCMPCCs.

\section{Conclusions}

Mathematical programs with second-order cone complementarity constraints (SOCMPCCs) present significant theoretical and computational challenges due to the inherent failure of standard constraint qualifications, such as Robinson's constraint qualification, at all feasible points. This failure complicates the direct application of classical nonlinear programming theories and necessitates specialized approaches for both theoretical analysis and numerical algorithms. While existing research has predominantly focused on theoretical advancements, including the characterization of normal cones, C-/M-/S-/K-stationary conditions, and various constraint qualifications, numerical methods for SOCMPCCs remain underdeveloped and largely unexplored.

In this study, we propose the augmented Lagrangian method (ALM) for SOCMPCCs. Our primary contributions and findings are summarized as follows.
We develop a tailored ALM algorithm framework specifically adapted to SOCMPCCs. By leveraging the matrix-free implementation and local convergence guarantees of ALM, our method effectively handles the intricate interplay of complementarity constraints and second-order cone constraints.
Under bounded ALM penalty parameter settings, we prove that the sequences generated by the algorithm converge to feasible points of the SOCMPCC. Furthermore, by introducing { SOCMPCC-weak-Robinson's CQ} or assuming the boundedness of the penalty parameter sequence, we establish convergence to K-stationary points, a critical optimality condition for SOCMPCCs. These results extend the classical convergence theory of ALM to the non-smooth and non-convex setting of SOCMPCCs.
Through comprehensive numerical experiments, we demonstrate the practical effectiveness of the proposed algorithm. The results confirm that our ALM framework achieves feasible solutions with high accuracy and robustness, even for problems where traditional methods struggle due to constraint qualification failures.

Our study bridges the gap between theoretical advancements and numerical solvers for SOCMPCCs. By integrating insights from Jordan algebra techniques and modern ALM variants, we provide a unified approach that combines theoretical rigor with computational practicality. Future directions include extending the framework to handle higher-dimensional cones and exploring applications in engineering and economic systems.

\section*{Conflict of Interest}
The authors declare that they have no conflicts of interest.

\section*{Acknowledgments}
The authors are grateful to the editor and two anonymous referees for their comments, which have helped us to improve the paper substantially.

\end{document}